\documentclass[reqno, 12pt]{amsart}

\usepackage[utf8]{inputenc}
\usepackage[T1]{fontenc}
\usepackage{lmodern}
\usepackage{amsmath}
\usepackage{amsthm}
\usepackage{amssymb}
\usepackage{latexsym}
\usepackage{amsmath,amsfonts,amssymb}
\usepackage{upgreek}
\usepackage{enumerate}
\usepackage{mathrsfs}
\usepackage{mathscinet}
\usepackage{stmaryrd} % crochets nombres entiers
\usepackage{graphicx}
\usepackage{xcolor}
\usepackage[all,cmtip]{xy}
\usepackage{url}
\usepackage{hyperref}
\usepackage[a4paper, hmargin=3cm, vmargin=3cm]{geometry}

\newcommand\mathens[1]{\mathbb{#1}} %fonte des ensembles classiques
\newcommand\todo[1]{\textcolor{red}{#1}} %à commenter si package plus adapté

\newcommand{\ud}{\mathrm{d}}

\newcommand{\N}{\mathens{N}}
\newcommand{\Z}{\mathens{Z}}

\newcommand{\R}{\mathens{R}}
\newcommand{\C}{\mathens{C}}

\newcommand{\T}{\mathens{T}}

\newcommand{\CP}{\C\mathrm{P}}

\newcommand{\RP}{\R\mathrm{P}}

\newcommand\sphere[1]{\mathens{S}^{#1}}

\DeclareMathOperator{\supp}{supp}

\newcommand{\ham}{\mathrm{Ham}}
\newcommand{\hamc}{\mathrm{Ham}_{c}}
\newcommand{\cont}{\mathrm{Cont}}
\newcommand{\conto}{\mathrm{Cont}_{0}}
\newcommand{\contoc}{\mathrm{Cont}_{0}^\mathrm{c}}

\newcommand{\id}{\mathrm{id}}
\newcommand{\nlmas}\upmu

\DeclareMathOperator\ind{ind}

\DeclareMathOperator\fix{Fix}

\newtheorem{thm}{Theorem}[section]
\newtheorem{lem}[thm]{Lemma}
\newtheorem{cor}[thm]{Corollary}
\newtheorem{prop}[thm]{Proposition}

\newtheorem{prop-def}[thm]{Definition-proposition}
\newtheorem{conj}[thm]{Conjecture}

\theoremstyle{definition}
\newtheorem{definition}[thm]{Definition}

\theoremstyle{remark}

\newtheorem{exs}[thm]{Examples}
\newtheorem{rem}[thm]{Remark}

\newcommand\tpsi{\widetilde{\psi}}
\newcommand\tphi{\widetilde{\varphi}}

\newcommand\Leg{\mathcal{L}}
\newcommand\uLeg{\widetilde{\Leg}}
\newcommand\tLambda{\widetilde{\Lambda}}
\newcommand\tPhi{\widetilde{\Phi}}
\newcommand\tPsi{\widetilde{\Psi}}

\newcommand\cleq{\preceq}

\newcommand\spec{\mathrm{Spec}}

\DeclareFontFamily{U}{mathb}{\hyphenchar\font45}
\DeclareFontShape{U}{mathb}{m}{n}{
      <5> <6> <7> <8> <9> <10> gen * mathb
      <10.95> mathb10 <12> <14.4> <17.28> <20.74> <24.88> mathb12
}{}
\DeclareSymbolFont{mathb}{U}{mathb}{m}{n}
\DeclareMathSymbol{\cll}{3}{mathb}{"CE}

\makeatletter\let\@wraptoccontribs\wraptoccontribs\makeatother

\title{Contact non-squeezing in various closed prequantizations}
\author[P.-A. Arlove]{Pierre-Alexandre Arlove}
\address{P.-A. Arlove, Universit\'e de Strasbourg, IRMA UMR 7501, F-67000 Strasbourg, France}
\email{paarlove@unistra.fr}


\begin{document}

\maketitle

\begin{abstract}
We describe some contact non-squeezing phenomena, first in lens spaces then in strongly orderable closed prequantizations in the sense of Liu \cite{Liu2020}. To detect and quantify these non-squeezing phenomena we define and compute two different contact capacities. In the former case the contact capacity comes from the spectral selectors constructed by Allais, Sandon and the author by the means of generating functions and Givental’s non-linear Maslov index \cite{allais2024spectral}. In the latter case the contact capacity comes from the order spectral selectors constructed by Allais and the author \cite{allais2023spectral}.

    \end{abstract}

\section{Introduction}
Gromov's non-squeezing theorem \cite{Gro85} is one of the cornerstones, if not the cornestone,  of symplectic geometry. Indeed, while it had been known for a long time that a symplectic manifold does not carry any local invariant, the existence of global invariants finer than the volume was not clear until Gromov's work in \cite{Gro85} (see also \cite{gromov1986partial,Eliashbergwave}). Originally stated for the standard symplectic Euclidean space, this theorem has then been generalized to all symplectic manifolds by Lalonde-McDuff \cite{lalondemcduff}.

For contact geometry, considered sometimes as the odd-dimensional analogue of symplectic geometry, the story is somewhat different. In this geometry, the volume is not an invariant anymore and, depending on the contact manifold considered, global invariants might not exist at all. For example, in the standard contact Euclidean space, which is the local model of contact geometry, any bounded domain can be contactly squeezed inside any arbitrary small neighborhood. Therefore, in a general contact manifold, non-squeezing phenomena cannot be expected at small scale. Nonetheless, 20 years after Gromov's result Eliashberg, Kim and Polterovich \cite{EKP} overcame this additional difficulty and proved a non-squeezing theorem at large scale in the prequantization of the standard symplectic Euclidean space. Unlike symplectic geometry, generalization of non-squeezing phenomena to other contact manifolds remain however rare. To the knowledge of the author, such phenomena have been detected only in prequantization of Liouville manifolds (\cite{albers,CantNS,camel,San11,FraserNS,ChiuNS,fraser2023contactnonsqueezinglargescale}), some exotic contact spheres \cite{uljarevicNS} and the unitary cotangent of the torus endowed with its canonical contact structure (see for example \cite[Thm 2.4.1]{shapeeliashberg}).

In this paper we show that interesting non-squeezing phenomena happen also in lens spaces endowed with their standard contact structure and in strongly orderable closed prequantizations.\\

\subsection{Non-squeezing in contact lens spaces}\label{sec : intro ns lens space}\

\noindent
Let us fix $n\in\N_{>0}$ a positive integer. Endow the Euclidean space $\R^{2n+2}$ with the coordinate functions $((x_1,y_1),\cdots,(x_{n+1},y_{n+1}))$ and let $\sphere{2n+1}:=\left\{\sum\limits_{j=1}^{n+1}x_j^2+y_j^2=1\right\}$ be the standard Euclidean sphere of radius $1$ centered at $0$. Identifying $\R^{2n+2}$ with $\C^{n+1}$ via $\Phi : ((x_1,y_1),\cdots,(x_{n+1},y_{n+1})) \mapsto (x_1+iy_1,\cdots,x_{n+1}+iy_{n+1})$ we get an action of $\sphere{1}:=\{e^{i\theta}\ |\ \theta\in\R\}\subset \C$ on $\sphere{2n+1}\subset \C^{n+1}$ given by multiplication. For all integer $k\geq 1$, the lens space $L_k^{2n+1}$ is the quotient space $\sphere{2n+1}/\Z_k$ where $\Z_k:=\Z/k\Z\simeq\{e^{\frac{2i\pi m}{k}}\ |\ m\in[0,k-1]\cap\N\}$ is identified with the $k$th roots of unity. It carries a contact distribution $\xi_k$ given by the kernel of the contact form $\alpha_k$ defined as the pullback, by the covering map $\sphere{2n+1}\to L_k^{2n+1}$, of $\alpha_1:=i_{\sphere{2n+1}}^*\left(\sum\limits_{j=1}^{n+1} (x_jdy_j-y_jdx_j)\right)$ where $i_{\sphere{2n+1}} : \sphere{2n+1}\hookrightarrow \R^{2n+2}$ is the inclusion. 

%For later purpose we denote by $J\in\mathrm{GL}(\R^{2n+2})$ the corresponding linear complex structure of $\R^{2n+2}$, i.e. $\Phi^{-1}\circ J \circ \Phi=i$. 

Finally we identify the complex projective space $\CP^n$ with the quotient space $\sphere{2n+1}/\sphere{1}$ and denote by $\pi_1 : \sphere{2n+1}\to\CP^n$ the projection.  We consider on $\CP^n$ the only symplectic form $\omega$ that satisfies $\pi_1^*\omega:=\ud \alpha_1$. Note that $\omega=2\omega_{\mathrm{FS}}$, where $\omega_{\mathrm{FS}}$ denotes the standard Fubini-Study form, and that $\pi_k^*\omega=\ud\alpha_k$, where $\pi_k : L_k^{2n+1}\to\CP^n$ is the natural projection. \\

%The main result of this note is about showing that some domains of lens spaces, for $k$ large enough, satisfy some contact non-squeezing properties, i.e. domains that can be squeezed into each other by an isotopy of diffeomorphisms but that cannot by any isotopy of contactomorphisms. The domains we consider are similar to the ones considered in \cite{EKP}, \cite{albers} which are of the form $U_r:=\mathcal{B}(r)\times S^1\subset \left(W\times S^1,\ker(dz-\lambda)\right)$, where $(W,\omega=d\lambda)$ is a Liouville manifold and $\mathcal{U}_{r}$ is the image of some starshaped domain $\mathcal{U}:=\mathcal{U}_1$ under the Liouville flow at time $\ln(r)$ for $r\in(0,1]$. Indeed lens spaces are also contact manifold coming from prequantization of a symplectic manifold : the complex projective space. However in our case, the symplectic manifold $\CP^{n}$ is compact and therefore does not admit a global Liouville vector field (see also Remark \ref{rem : rem 1} for more examples of contact non-squeezing phenomena in compact setting). We will overcome this difficulty by the most trivial approach : proceed in Darboux chart in the symplectic manifold.\\  

$M$ being a smooth manifold and $\xi$ a contact distribution on $M$, we say that a subset $U_1\subset M$ cannot be contactly squeezed inside another subset $U_2\subset M$ if for all contactomorphisms $\phi\in\contoc(M,\xi)$ isotopic to the identity through a compactly supported contact isotopy, $\phi(U_1)$ does not lie inside $U_2$, i.e. $\phi(U_1)\cap U_2\ne\phi(U_1)$. 

Denoting by $\mathbb{B}^{2n}(r):=\{((x_1,y_1),\cdots,(x_n,y_n))\in\R^{2n}\ |\ \sum\limits_{j=1}^n x_j^2+y_j^2<r^2\}$ the standard open Euclidean ball of radius $r>0$ in $\R^{2n}$ centered at $0$ and by $\omega_0:=\sum\limits_{j=1}^n dx_j\wedge dy_j$ the standard symplectic form of $\R^{2n}$ we will prove the following:

%Let us consider $\Psi : (\mathbb{B}^{2n}(R),\omega_0)\hookrightarrow (\CP^n,\omega)$ a symplectic embedding, for some $R>0$, and for all $0<r\leq R$ denote by $\mathcal{B}(r):=\Psi(\mathbb{B}^{2n}(r))$ and by $B_k(r):=\pi_k^{-1}\mathcal(\mathcal{B}(r))\subset L_k^{2n+1}$. From this construction it is easy to see that for all $r_1,r_2>0$ there exists a diffeomorphism $\phi$ of $L_k^{2n+1}$ which is isotopic to the the identity and that sends $U_{r_2}^k$ inside $U_{r_1}^k$, i.e. $\phi(U_{r_2}^k)\subset U_{r_1}^k$ (see Remark \ref{rem : smooth squeezing} for more details). However the story is different for contactomorphisms

\begin{thm}\label{thm : non-squeezing optimal}
  Let $\Psi_1 :(\mathbb{B}^{2n}( R_1),\omega_0)\hookrightarrow (\CP^n,\omega)$ and $\Psi_2 :(\mathbb{B}^{2n}(\lambda R_2),\omega_0)\hookrightarrow (\CP^n,\omega)$ be two symplectic embeddings for some $R_1,\ R_2>0$, $\lambda>\sqrt{3}$ and let $a_1\in (0, R_1],\ a_2\in (0,R_2]$ be two positive real numbers. For any integer $k\geq 2$, if there exists $j\in\N_{>0}$ such that
    \[ \pi a_1^2> \frac{2\pi}{k}\cdot j \geq \pi a_2^2,\]  
    then $B^1_k(a_1)$ cannot be contactly squeezed inside $B^2_k(a_2)$, 
    where by definition $B^i_k(a_i):=\pi_k^{-1}(\Psi_i(\mathbb{B}^{2n}(a_i)))\subset (L_k^{2n+1},\xi_k)$ for all $i\in\{1,2\}$.

      % \item For all $A,a\in(0,R]$ the open set $B_k(A)$ cannot be contactly squeezed inside $B_k(a)$ if there exists $j\in\N_{>0}$ satisfying
   %    \begin{equation}
     %      \pi A^2>\frac{2\pi}{k}j\geq\pi a^2. 
     %  \end{equation}

%\end{enumerate}
    
\end{thm}

\begin{rem}\label{rem : introduction}
    \begin{enumerate}
   \item It might be that Theorem \ref{thm : non-squeezing optimal} remains true for some $\lambda\leq\sqrt{3}$.  However, for the purpose of our proof, $\lambda$ has to be strictly greater than $1$ since, as we will see in Subsection \ref{sec : computation upper bound}, we need to displace $\mathbb{B}^{2n}(a_2)$ from itself inside $\mathbb{B}^{2n}(\lambda R_2)$. 
   
  % More precisely one can state the above Theorem \ref{thm : non-squeezing optimal} for a symplectic embedding $\Psi : (\mathbb{B}^{2n}(\lambda R),\omega_0)\hookrightarrow \CP^n$ and $A\in (0,\lambda R]$,  $a\in(0,R]$ with  $\lambda< 2$\footnote{\todo{$\lambda=\sqrt{2}$ or $\sqrt{3}$ optimal?}}. 

        \item These non-squeezing phenomena do not happen if we forget about the contact structure. Indeed it is easy to see that if $R_1=\lambda R_2=:R$ and $\Psi_1=\Psi_2=:\Psi$ then  for all $r_1,\ r_2\in(0,R]$, there exists a diffeomorphism of $L_k^{2n+1}$ which is isotopic to the the identity - through diffeomorphisms - and that sends $B_k(r_1)$ inside $B_k(r_2)$ where $B_k(r):=\pi_k^{-1}(\Psi(\mathbb{B}^{2n}(r)))$ for all $r\in (0,R]$.
%        \item Theorem \ref{thm : non-squeezing optimal} ensures that fixing any symplectic embedding $\Psi :(\mathbb{B}^{2n}(\lambda R),\omega_0)\hookrightarrow (\CP^n,\omega)$ there exists contact non-squeezing phenomena for the family of domains $\{B_k(r):=\Psi(\mathbb{B}^{2n}(r))\}$  in $(L_k^{2n+1},\xi_k)$ whenever $k$ is big enough. Indeed for all $R>0$ there exists $k_0$ so that $\pi(\lambda R)^2>\frac{2\pi}{k}$ for all $k\geq k_0$.
        \item In \cite[Thm 1 and Prop 2.11]{Karshon_2005}\footnote{The conventions used in \cite{Karshon_2005} are different: they consider the symplectic form $\widetilde{\omega}:=\frac{1}{2\pi}{\omega}$ and show that there exists an embedding of $(\widetilde{\mathbb{B}^{2n}}(1),\omega_0)\hookrightarrow (\CP^n,\widetilde{\omega})$, where $\widetilde{\mathbb{B}^{2n}}(a):=\mathbb{B}^{2n}(\sqrt{\frac{a}{\pi}})$}  Karshon and Tolman construct a symplectic embedding $\Psi :(\mathbb{B}^{2n}(R),\omega_0)\hookrightarrow (\CP^n,\omega)$, for $R=\sqrt{2}$, and show that it is optimal in the sense that such symplectic embeddings do not exist for $R>\sqrt{2}$ (see also \cite[Thm 1]{Loi_2015}). Considering this optimal embedding, one can see that non-squeezing phenomena appear for all lens spaces, even for the real projective space $(\RP^{2n+1},\xi):=(L_2^{2n+1},\xi_2)$, since 
$\pi \left(\sqrt{2}\right)^2>\frac{2\pi}{2}$.
     %   \item More generally for any integer $k\geq 2$ and n-tuple $\underline{w}:=(w_1,\cdots,w_n)$ of positive integers relatively prime to $k$ we have a free action of $\Z_k(\underline{w})\simeq \{ (e^{\frac{2\pi i m w_1}{k}}{k},\cdots,e^{\frac{2\pi i m w_n}{k}})\ |\ m\in [0,k-1]\cap \N\}$ on $\sphere{2n+1}$. The quotient space $L_k^{2n+1}(\overline{w}):=\sphere{2n+1}/\Z_k(\overline{w})$ also carries a contact form $\alpha_k(\overline{w})$ given by $(\pi_1^k(\overline{w})^*\alpha_k(\overline{w})=\alpha_1$ where $\pi_1^k(\overline{w}) : \sphere{2n+1}\to L_k^{2n+1}(\overline{w})$ denotes the covering. Its $\alpha_k(\overline{w})$-Reeb flow is periodic of minimal period $T_{\overline{w}}\in\{\frac{2 m \pi}{k}\ |\ m\in[1,k]\cap\N\}$. When  all the weights are equal then the Reeb flow is Zoll of minimal period $\frac{2\pi}{k}$. 

      %  There exists a symplectic embedding $\Phi :  (\mathbb{B}^{2n}(R),\omega_0)\hookrightarrow (\CP^n,\omega)$\footnote{In \cite{Karshon_2005} the conventions are different, they consider the symplectic form $\widetilde{\omega}:=\frac{1}{2\pi}{\omega}$ and show that there exists an embedding of $(\widetilde{\mathbb{B}^{2n}}$
     % \item \todo{Parler du cas de la sphère (squeezing ET non-squeezing), du cotangent unitaire du tore et de ses cousins  en intro}
    %   \item \todo{peut-etre dire qu'on peut avoir le meme theoreme en etudiant differents symplectic embedding}
 \end{enumerate}
\end{rem}

We will prove Theorem \ref{thm : non-squeezing optimal} by constructing a contact capacity defined on open subsets of $(L_k^{2n+1},\xi_k)$ and by estimating this contact capacity for the lift of Darboux balls. This contact capacity comes from the spectral selectors defined on $\widetilde{\conto}(L_k^{2n+1},\xi_k)$ in \cite{allais2024spectral} by the means of generating functions and Givental's non-linear Maslov index (see Section \ref{sec : spectral selectors for lens spaces}).

\subsection{Non-squeezing in strongly orderable prequantizations}\label{sec : intro ns prequantization}
Recall that if $(M,\xi=\ker\alpha)$ is a closed cooriented contact manifold then $(M\times M\times\R,\Xi:=\ker(\beta:=\mathrm{pr}_2^*\alpha-e^\theta\mathrm{pr}_1^*\alpha))$, where $\mathrm{pr}_i :M\times M\times\R\to M,\ (x_1,x_2,\theta)\mapsto x_i$, is a cooriented contact manifold and the diagonal $\Delta:=\{(x,x,0)\ |\ x\in M\}$ is a Legendrian lying inside it. We will call the latter manifold the contact product of $(M,\xi)$\footnote{the contact structure on the contact product of $(M,\xi)$ does not depend on the choice of a contact form defining the contact distribution $\xi$ up to isomorphism} with itself and will call the $1$-form $\beta$ the product contact form associated to $\alpha$.   Using the terminology of \cite{Li2021,CCR} we say that  $(M,\ker\alpha)$ is strongly orderable if there does not exist any positive contractible loop of Legendrians based at $\Delta$. In particular, if $(M,\ker\alpha)$ is strongly orderable then it is orderable in the sense of \cite{EP00}, i.e.  there does not exist any positive contractible loop of contactomorphisms in $\conto(M,\ker\alpha)$.  See Sections \ref{sec : orderability and periodic reeb flow} and \ref{sec : strongly orderability}  for more details about these notions. Colin-Chantraine-Dimitroglou Rizell show in \cite[Thm 1.19]{CCR} that if $(M,\xi)$ is a closed hypertight contact manifold, i.e. there exists a contact form $\alpha$ supporting $\xi$ whose Reeb flow does not have any contractible Reeb orbit, then $(M,\xi)$ is strongly orderable. They also show that any closed contact manifold admitting a Liouville filling whose symplectic homology does not vanish is strongly orderable \cite[Thm 1.18]{CCR}.  \\

We say that a closed contact manifold $(M_1,\xi_1')$ is a prequantization over a closed  symplectic manifold $(W,\omega')$ if there exists a $\mathbb{T}^1$-principal bundle $\pi_1' : M_1\to W$, where $\mathbb{T}^1=\R/2\pi\Z$, and a contact form $\alpha_1'$ supporting $\xi_1'$ satisfying ${\pi_1'}^*\omega'=\ud \alpha'_1$ and the Reeb flow of which, is Zoll\footnote{all its orbits are periodic with the same minimal period} of minimal period $2\pi$, induces the $\mathbb{T}^1$-action. 

Suppose that $(M_1,\xi_1')$ is a closed prequantization over $(W,\omega')$, then for any $k\in\N_{>0}$ we denote by $M_k:=M_1/\Z_k$, where $\Z_k$ is seen as a subgroup of $\mathbb{T}^1$ via the map $[j]\mapsto \frac{2\pi j}{k}$, and $\xi'_k$ the contact distribution on $M_k$ given by the kernel of the unique $1$-form $\alpha'_k$ whose pull back via the covering $M_1\to M_1/\Z_k$ is $\alpha'_1$. Finally we denote by $\pi'_k : M_k\to W$ the natural projection. Note that ${\pi'_k}^*\omega'=\ud\alpha'_k$ and that $(M_k,\xi'_k)$ is a prequantization over $(W,k\omega')$ by taking $k\alpha'_k$ to be the contact form inducing the Zoll Reeb flow of minimal period $2\pi$.

%Consider $M_1$ a closed manifold endowed with a contact distribution $\xi'_1$ given by the kernel of a $1$-form $\alpha'_1$ whose Reeb flow $(\phi_{\alpha'_1}^t)_{t\in\R}$ is Zoll of minimal period $2\pi$, i.e. for all $x\in M_1,\ t\in\R$, the equality $\phi_{\alpha'_1}^t(x)=x$ holds if and only if $t$ is a multiple of $2\pi$. The Reeb flow therefore induces a $\mathbb{T}^1:=\R/2\pi\Z$ action on $M_1$ and it is a well known fact, see for example \cite[Chap. 7]{geiges}, that the quotient manifold $W:=M_1/\mathbb{T}^1$ is endowed with a symplectic form $\omega'$ uniquely defined by ${\pi'_1}^*\omega'=\ud\alpha'_1$ where $\pi'_1 : M_1\to W$ denotes the projection . We say that $(M_1,\ker\alpha'_1)$ is a prequantization over $(W,\omega')$. 

\begin{thm}\label{thm : ns for prequantization} Suppose that $(M_1,\xi'_1)$ is a strongly orderable prequantization over a closed symplectic manifold $(W,\omega')$  and
let $\Psi_1 :(\mathbb{B}^{2n}( R_1),\omega_0)\hookrightarrow (W,\omega')$, $\Psi_2 :(\mathbb{B}^{2n}(\lambda R_2),\omega_0)\hookrightarrow (W,\omega')$ be two symplectic embeddings for some $R_1\in(0,\sqrt{2}]$, $R_2>0$, $\lambda>\sqrt{3}$, and $a_1\in (0, R_1],\ a_2\in (0,R_2]$ be two positive real numbers. For any $k\in\N_{>0}$, if there exists $j\in\N_{>0}$ such that
    \[ \pi a_1^2> \frac{2\pi}{k}\cdot j \geq \pi a_2^2,\]  
    then $B^1_k(a_1)$ cannot be contactly squeezed inside $B^2_k(a_2)$, 
    where by definition $B^i_k(a_i):={\pi'_k}^{-1}(\Psi_i(\mathbb{B}^{2n}(a_i)))\subset (M_k,\xi'_k)$ for all $i\in\{1,2\}$. 
\end{thm}

As for Theorem \ref{thm : non-squeezing optimal} we prove Theorem \ref{thm : ns for prequantization} by constructing a contact capacity defined on open subsets of $(M_k,\xi'_k)$ and by estimating this contact capacity for the lift of Darboux balls. The contact capacity this time will come from the order spectral selectors defined in \cite{allais2024spectral} on the universal cover of the isotopy class of the Legendrian $\Delta_k$ lying in the contact product of $(M_k,\xi'_k)$ with itself.

\begin{rem}\label{rem : deuxieme rem intro}\
\begin{enumerate}

\item In Theorem \ref{thm : ns for prequantization} we make the assumption that $R_1\leq\sqrt{2}$ so that $\pi R_1^2$ is not greater than the minimal period of the Reeb flow associated to the contact form $\alpha'_1$ satisfying ${\pi'_1}^*\omega'=\ud \alpha'_1$ (see the 4th step of the proof of Proposition \ref{prop : lower bound for strongly orderable prequantization} for more details). This technical assumption was not made in Theorem \ref{thm : non-squeezing optimal} because,  by Theorem \cite[Thm 1]{Karshon_2005} discussed in Remark \ref{rem : introduction}, it is automatically satisfied. In particular, when $(M_1,\xi'_1)$ is a strongly orderable prequantization, one can expect to detect non-squeezing phenomena in $(M_k,\xi'_k)$ thanks to Theorem \ref{thm : ns for prequantization} only for $k>1$.
\item Even if the lens space $(L_k^{2n+1},\xi_k)$ defined in Section \ref{sec : intro ns lens space} can be seen as a prequantization over $(\CP^n,k\omega=2k\omega_{FS})$, one cannot deduce Theorem \ref{thm : non-squeezing optimal} from Theorem \ref{thm : ns for prequantization}. First, it is not clear to the author whether contact lens spaces are strongly orderable or not in general. Indeed $(L_k^{2n+1},\xi_k)$, for $ n>1$, is not hypertight since it satisfies the Weinstein conjecture (see \cite{granja2014givental}), i.e. the Reeb flow of any contact form $\alpha$ satisfying $\ker\alpha=\xi_k$ admits a periodic orbit, and its fundamental group $\pi_1(L_k^{2n+1})$ is finite. Moreover, if $n\in\N_{>0}$ is not a power of $2$, Zhou \cite{Zhou_2021} showed that $(L_2^{2n-1},\xi_{2})$ is not Liouville fillable. Second, for the reasons discussed in the first point of this Remark \ref{rem : deuxieme rem intro} and the fact that $(L_1^{2n+1},\xi_1)=(\sphere{2n+1},\xi_1)$ is not strongly orderable for $n\geq 1$ (see \cite{EKP}), Theorem \ref{thm : ns for prequantization} cannot detect non-squeezing phenomena in $(L_p^{2n+1},\xi_p)$ when $p$ is a prime number  even when it is known that the latter is strongly orderable. In particular, $(L_2^3,\xi_2)$ is strongly orderable, since the standard symplectic codisk bundle of the $2$-sphere is a Liouville filling whose symplectic homology does not vanish (see Example \ref{exemple} below), but from what we have just said Theorem \ref{thm : ns for prequantization} cannot detect non-squeezing phenomena in the latter contact manifold while Theorem \ref{thm : non-squeezing optimal} does (see Remark \ref{rem : introduction}).

%In particular, it cannot detect non-squeezing phenomena in $(L_2^3,\xi_2)$ which is known to be strongly orderable while thanks to Rem ... Thm allows us to do so ...  the lens spaces we are considering are strongly orderable, Theorem \ref{thm : ns for prequantization} does not allow to recover fully Theorem \ref{thm : non-squeezing optimal}. For instance, the real projective space of dimension $3$ with its standard contact structure, which we have denoted by $(L_2^3,\xi_2)$ in Section \ref{sec : intro ns lens space}, is strongly orderable thanks to \cite[Thm. 1.18]{CCR}. Indeed, it is contactomorphic to the standard contact unit cotangent of the $2$-sphere which is obviously Liouville fillable by the standard symplectic disk cotangent bundle and it is well known that the symplectic homology of the latter does not vanish (see for instance the introduction of \cite{CCR} or \cite[Example 1.6]{albers} combined with \cite[Thm. 13.3]{Ritter_2013}). On top of that, it is a $2\pi$-prequantization over $(\CP^n,2\omega=4\omega_{FS})$. While Theorem \ref{thm : non-squeezing optimal} allows us to detect non-squeezing phenomena in   $(L_2^3,\xi_2)$ (see Remark \ref{rem : introduction}) for the reason we have just discussed in the first point of this Remark \ref{rem : deuxieme rem intro} and the fact that $2$ is a prime number, Theorem \ref{thm : ns for prequantization} would not allow us to do so.

\item We conjecture that Theorem \ref{thm : ns for prequantization} also holds if we assume $(M_1,\xi'_1)$ to be orderable in the sense of \cite{EP00}, i.e. there does not exist any positive contractible loop of contactomorphisms, which is \textit{a priori} weaker than to assume $(M_1,\xi'_1)$ to be strongly orderable. The validity of this conjecture would be an easy consequence of the proof given below in Section \ref{sec : Contact capacity for strongly orderable prequantizations} together with the validity of the conjecture made in \cite{allais2023spectral} about the spectrality of the (pseudo-)spectral selectors for contactomorphisms.
\item It seems  very plausible that with the tools developped in \cite{cant2024remarkseternalclassessymplectic} one could derive similar non-squeezing phenomena satisfied by the ideal contact boundary of a semipositive and convex-at-infinity symplectic manifold $W$ whose unit in symplectic homology is not eternal.
\end{enumerate}
\end{rem}

\begin{exs}\label{exemple}\
\begin{enumerate}

\item For any closed manifold $X$, its associated cosphere bundle with its standard contact structure is strongly orderable by \cite[Thm. 1.18]{CCR}, since it is Liouville fillable by the standard symplectic codisk bundle and the symplectic homology of the latter does not vanish (see for instance the introduction of \cite{CCR} or \cite[Example 1.6]{albers} combined with \cite[Thm. 13.3]{Ritter_2013}). If we assume moreover that $X$ admits a Riemaniann metric $g$ so that all its geodesics are closed with length $2\pi$ (see \cite[Chap. 3.4]{Marco} for examples) then the cosphere bundle is, in addition to be strongly orderable, also a prequantization (see \cite[Thm. 7.2.5]{geiges}).

\item If $(W,\omega')$ is a closed symplectic manifold so that $\left<[\omega'],H_2(W,\Z)\right>\subset 2\pi\Z$ and $\pi'_1 : M_1\to W$ denotes the $\R/2\pi\Z$-principal bundle whose Euler class, after taking its image by the map $H^2(W,\Z)\to H^2(W,\R)\simeq H^2_{\mathrm{dR}}(W,\R)$, is equal to $-[\frac{\omega'}{2\pi}]$, then $M_1$ carries a contact form $\alpha'_1$ that turns $(M_1,\ker\alpha'_1)$ into a prequantization over $(W,\omega')$, i.e. ${\pi_1'}^*\omega'=d\alpha_1'$, the Reeb flow of $\alpha_1'$ is Zoll of minimal period $2\pi$ and it induces the $\R/2\pi\Z$ action (see for instance \cite[Chap 7.2]{geiges}). If we suppose in addition that $(W,\omega')$ is symplectically aspherical, i.e. $\left<[\omega'], H_2^S(M,\Z)\right>=0$ where $H_2^S(M,\Z)$ is the image of the Hurewitz homomorphism $\pi_2(M)\to H_2(M,\Z)$, then $(M_1,\xi'_1)$ is hypertight and in particular is strongly orderable thanks to \cite[Thm. 1.19]{CCR}. Therefore we get non-squeezing phenomena thanks to Theorem \ref{thm : ns for prequantization} in $(M_k,\xi'_k)$ for $k>1$ big enough.

\end{enumerate}

\end{exs}

\subsection{Discussion} It is interesting to note that the $\alpha_k$, and $\alpha'_k$ Reeb flows on $(L_k^{2n+1},\xi_k)$ and on $(M_k,\xi'_k)$, as described above in Sections \ref{sec : intro ns lens space} and \ref{sec : intro ns prequantization}, are periodic with period $\frac{2\pi}{k}$ for all integers $k\geq 1$. Indeed, the  periodicity of the Reeb flow often plays a crucial role in the existence of non trivial global invariants and this article is no exception to this fact (see Lemma \ref{lem : prequantization}). Furthermore, the global invariants one can deduce from Theorem \ref{thm : non-squeezing optimal} and Theorem \ref{thm : ns for prequantization} naturally take value in the discrete set consisting in multiples of this period. Similarly, in the contact manifold $\left(\R^{2n}\times\R/\Z,\ker(\alpha_0:=\ud z-\frac{1}{2}\sum\limits_{i=1}^n (x_idy_i-y_idx_i))\right)$ where the first contact non-squeezing phenomena was described by Eliashberg, Kim and Polterovich in \cite{EKP}, the  $\alpha_0$-Reeb flow is $1$-periodic, and they show that it is impossible to contactly squeeze $\mathbb{B}^{2n}(A)\times\R/\Z$ inside $\mathbb{B}^{2n}(a)\times\R/\Z$ if there exists an integer $j\in\N_{>0}$ so that $A>j\geq a$.

The same kind of phenomenon can be observed when studying bi-invariant metrics on the group $\contoc(M,\xi)$ or its universal cover $\widetilde{\contoc}(M,\xi)$. While Sandon in \cite{sandonmetrique} constructs an unbounded bi-invariant metric on $\contoc(\R^{2n}\times\R/\Z,\ker\alpha_0)$ taking integer values, the authors in \cite{allais2024spectral} construct an unbounded bi-invariant metric on $\widetilde{\conto}(L_k^{2n+1},\xi_k)$ taking values in $\frac{2\pi}{k}\N$\footnote{It is shown in \cite{FPR} that, for any contact manifold $(M,\xi)$, a bi-invariant metric on $\contoc(M,\xi)$ (resp. $\widetilde{\contoc}(M,\xi)$) has to induce the discrete topology on $\contoc(M,\xi)$ (resp. on $\widetilde{\contoc}(M,\xi)\setminus\pi_1(\contoc(M,\xi))$)}.\\

%Moreover, it is shown in \cite{BIP} that bi-invariant metrics on $\contoc(M,\xi)$ have to be discrete, i.e. they induce the discrete topology on $\contoc(M,\xi)$, regardless of the contact manifold considered. Therefore, if not unbounded, a bi-invariant metric on this group is equivalent to the trivial metric.

%Acutally \todo{ce sont les memes fonctions qui permettent de définir les metriques et qui permettent de prouver le NS.}

What we have just described should be nuanced: beeing a contact manifold admitting a periodic Reeb flow is not a sufficient condition to have rigidity results and existence of global invariants. In fact, the $\alpha_1$-Reeb flow of the standard contact sphere $(\sphere{2n+1},\ker\alpha_1)$ is periodic but any two non-empty open domains that are not the whole sphere can be contactly squeezed inside each other. In the same way, it is shown in \cite{FPR} that, for $n\in\N_{>0}$, the groups $\conto(\sphere{2n+1},\ker\alpha_1)$ and $\widetilde{\conto}(\sphere{2n+1},\ker\alpha_1)$ are bounded, i.e. any bi-invariant metric on these groups has to be bounded. 

It seems that another important property that a contact manifold should have so that one could expect meaningful invariant measurements is to be orderable (see \ref{sec : orderability and periodic reeb flow}, \ref{sec : strongly orderability} for definitions). For $n>1$ and $k\geq 2$, the sphere $(\sphere{2n+1},\ker\alpha_1)$ and the lens space $(L_k^{2n+1},\ker\alpha_k)$ are perfect representative cases since the sphere is not orderable \cite{EKP} while the lens space is \cite{granja2014givental}, and, as we have just mentionned, interesting contact measurements exist for the latter but not for the former.
%While, for any $n\in\N_{>0}$ and any integer $k\geq 2$ , the lens space $(L_k^{2n+1},\ker\alpha_k)$ is orderable \cite{granja2014givental}, the standard contact sphere $(\sphere{2n+1},\ker\alpha_1)$ is not \cite{EKP}.
In the same direction, a similar observation can be made for bi-invariant metrics on contactomorphisms groups. On the one hand, to the knowledge of the author, the only examples in the litterature of closed contact manifolds $(M,\xi)$ admitting unbounded bi-invariant metrics on $\widetilde{\conto}(M,\xi)$ are the contact manifolds $(M,\xi)$ that are orderable and that admit a periodic Reeb flow \cite{FPR,discriminante,granja2014givental,allais2024spectral}. On the other hand, for a large family of closed non-orderable contact manifolds $(M,\xi)$ Courte and Massot show that $\conto(M,\xi)$ and $\widetilde{\conto}(M,\xi)$ have to be bounded \cite{courtemassot}. \\

% That being so, we discuss in Section \ref{sec : generalization} a potential generalization of Theorem \ref{thm : non-squeezing optimal} to orderable prequantizations of closed symplectic manifolds which are particular cases of orderable closed contact manifold with periodic Reeb flow. \\

 Actually, while convenient but not sufficient, the existence of a periodic Reeb flow is not even a necessary condition to have contact global invariants. For example, the unitary cotangent of the torus $\mathbb{P}^+T^*\T^n$ endowed with its canonical contact structure $\xi_{\mathrm{can}}$ does not admit any periodic Reeb flow when $n>1$ (see for example \cite{chernovnemirovski1}). But, as briefly mentionned above, one can use an interesting contact global invariant for domains in $(\mathbb{P}^+T^*\T^n,\xi_{\mathrm{can}})$, which is called the \textit{shape}\footnote{unlike the other invariant measurements discussed so far, the \textit{shape} can be used to construct “non-discrete” measurements}, to deduce non-squeezing phenomena and other rigidity results concerning this contact manifold \cite{shapeeliashberg,CantNS,müller2013gromovs,EP00}. It would therefore be interesting to investigate whether the groups $\widetilde{\conto}(\mathbb{P}^+T^*\T^n,\xi_{\mathrm{can}})$ and $\ \conto(\mathbb{P}^+T^*\T^n,\xi_{\mathrm{can}})$ are bounded or not (see also \cite[Prop. 4.14, Rem. 4.15]{allais2023spectral}).

\subsection*{Organization of the paper}   In Section \ref{sec : spectral selectors} we set some conventions and give basic definitions and properties of orderability and spectral selectors. In Section \ref{sec : spectral selectors for lens spaces} we briefly recall the construction and main properties of the spectral selectors on $\widetilde{\conto}(L_k^{2n+1},\xi_k)$ constructed in \cite{allais2024spectral}. In Section \ref{sec : Contact capacity} we define the above mentionned contact capacity for open domains of $(L_k^{2n+1},\xi_k)$ coming from the spectral selectors. In this same Section \ref{sec : Contact capacity} we state - without prooving - some properties of this contact capacity (see Theorem \ref{thm : contact capacity}) which allow us to give the proof of Theorem \ref{thm : non-squeezing optimal}. In Section \ref{sec : capacite-energy} we show a capacity-energy inequality satisfied by the contact capacity. In Section \ref{sec : computation of the capacity} we prove Theorem \ref{thm : contact capacity} by estimating and computing the contact capacity of domains that are lifts of Darboux balls to $(L_k^{2n+1},\xi_k)$ using the capacity-energy inequality of Section \ref{sec : capacite-energy} and some explicit constructions. In Section \ref{sec : spectral selectors for strongly orderable prequantization} we define the spectral selectors on $\widetilde{\conto}(M,\xi)$ when $(M,\xi)$ is a strongly orderable closed contact manifold. Finally in Section \ref{sec : Contact capacity for strongly orderable prequantizations} we use these spectral selectors to define a contact capacity for domains in strongly orderable prequantizations and prove Theorem \ref{thm : ns for prequantization}.

\subsection*{Acknowledgment} 
The author thanks Alberto Abbondandolo, Johanna Bimmermann, Stefan Nemirovski and Sheila Sandon for  interesting and stimulating discussions.

The author was
partially supported by the \emph{Deutsche Forschungsgemeinschaft} under the
\emph{Collaborative Research Center SFB/TRR 191 - 281071066 (Symplectic
Structures in Geometry, Algebra and Dynamics)}.

\section{Spectral Selectors}\label{sec : spectral selectors}

\subsection{Conventions and definition} Let $(M,\xi)$ be a closed coorientable contact manifold. We say that a $[0,1]$-family of contactomorphisms $\{\phi_t\}_{t\in[0,1]}\subset\cont(M,\xi)$, that we abbreviate with $(\phi_t)$, is a contact isotopy if the map $[0,1]\times M\to M$, $x\mapsto \phi_t(x)$ is smooth. We denote by $\conto(M,\xi)$ the group of contactomorphisms $\phi$ that are contact isotopic to the identity: there exists a contact isotopy $(\phi_t)$ such that $\phi_0=\id$ and $\phi_1=\phi$. We endow $\conto(M,\xi)$ with the $C^1$-topology, denote by $\widetilde{\conto}(M,\xi)$ its universal cover and $\Pi :\widetilde{\conto}(M,\xi)\to\conto(M,\xi)$ the projection. We often see $\widetilde{\conto}(M,\xi)$ as equivalence classes of contact isotopies starting at the identity where two isotopies represent the same equivalence class if they are homotopic relatively to endpoints. % We denote by $\widetilde{\id}\in\widetilde{\conto}(M,\xi)$ the neutral element (represented by the constant isotopy $(\phi_t\equiv \id)$).

A point $x\in M$ is said to be a discriminant of a contactomorphism $\phi\in\conto(M,\xi)$ if $\phi(x)=x$ and $(\phi^*\alpha)_x=\alpha_x$, for some - and thus all - contact form $\alpha$ supporting $\xi$, i.e. $\ker\alpha=\xi$. 

Let $\alpha$ be a contact form supporting $\xi$, i.e. $\ker\alpha=\xi$. We say that a point $x\in M$ is an $\alpha$-translated point of $\phi\in\conto(M,\xi)$ if there exists $T\in\R$ so that $x$ is a discriminant point of $\phi_{\alpha}^{-T}\circ\phi$, where $\phi_\alpha^s$ denotes the Reeb flow at time $s\in\R$ of $\alpha$, i.e. the flow of the vector field $R_\alpha$ uniquely defined by the two equations $\alpha(R_\alpha)\equiv 1$ and $\ud\alpha(R_\alpha,\cdot)\equiv 0$. For any $\phi\in\conto(M,\xi)$ we denote and define its $\alpha$-spectrum by \[\spec^\alpha(\phi):=\left\{t\in\R\ |\ \phi_\alpha^{-t}\circ \phi \text{ has a discriminant point}\right\}.\] 
Similarly, the $\alpha$-spectrum $\spec^\alpha(\tphi)$ of any
$\tphi\in\widetilde{\conto}(M,\xi)$ 
is by definition
\[\spec^\alpha(\tphi):=\spec^\alpha\left(\Pi(\tphi)\right).\]

%contactomorphism $\varphi\in\cont_0(M,\xi)$ (resp. $\tphi\in\widetilde{\conto}(M,\xi))$

\begin{definition}\label{def : spectral selector}
    We say that a map $c: \widetilde{\conto}(M,\ker\alpha)\to\R$ is a $\alpha$-spectral selector if it is $C^1$-continuous and if $c(\tphi)\in\spec^\alpha(\tphi)$.
\end{definition}

\subsection{Orderability, Periodic Reeb flow and invariance of spectral selectors}\label{sec : orderability and periodic reeb flow} Let $(M,\xi)$ be a closed coorientable contact manifold. From now on we fix a coorientation of $\xi$, i.e. a vector field $X\in \chi(M)$ so that $X(M)\subset TM\setminus\xi$.  We say that a contact form $\alpha$ supports $\xi$ if $\ker\alpha=\xi$ and $\alpha(X)>0$.

A contact isotopy $(\phi_t)$ is said to be non-negative (resp. positive) if for some, and thus all, contact form $\alpha$ supporting $\xi$ we have $\alpha\left(\frac{d}{dt}\phi_t(x)\right)\geq 0$ (resp. $>0$) for all $x\in M$. Following \cite{EP00} we define a binary relation on $\widetilde{\conto}(M,\xi)$ by saying that $\tphi\cleq\tpsi$ if there exists a path $(\tphi_t)$ joining them such that $(\Pi(\tphi_t))$ is a non-negative contact isotopy. It is shown in \cite{EP00} that this binary relation is a partial order if and only there does not exist any positive contractible loop of contactomorphisms, and in this case we say that $(M,\xi)$ is orderable. The following Lemma will be crucial in this paper and was, to some extent, already noticed by Sandon in \cite{San11}

\begin{lem}\label{lem : prequantization}
    Let $(M,\xi)$ be a orderable closed contact manifold and $\alpha$ be a contact form supporting $\xi$ whose Reeb flow is periodic of minimal period $T>0$. If $c : \widetilde{\conto}(M,\xi)\to\R$ is a $\alpha$-spectral selector which is monotone, i.e. $\tphi_1\cleq\tphi_2$ implies that $c(\tphi_1)\leq c(\tphi_2)$, then 
    \[\lceil c(\tpsi\cdot\tphi\cdot\tpsi^{-1})\rceil_{T}=\lceil c(\tphi)\rceil_T\quad \quad \text{ for all }\tphi,\ \tpsi\in\widetilde{\conto}(M,\xi).\]
\end{lem}
\begin{proof}
Consider $(\psi_t)$ a path representing $\tpsi$ and denote by $\tpsi_s:=[(\psi_{st})]$ for all $s\in[0,1]$. Let us first suppose that $\Pi(\tphi)$ does not have discriminant point. Thus $\Pi(\tpsi_s\cdot\tphi\cdot\tpsi_s^{-1})$ does not have discriminant point and so $c(\tpsi_s\cdot \tphi\cdot\tpsi_s^{-1})\notin T\Z$ for all $s\in[0,1]$. We then conclude using the continuity of $s\mapsto c(\tpsi_s\cdot\tphi\cdot\tpsi_s^{-1})$. We treat the general case using the fact that $\spec^\alpha(\tphi)$ is nowhere dense (see \cite[Lemma 2.11]{albers}) and therefore there exists a sequence $(\varepsilon_n)_{n\in \N}$ of positive real numbers going to $0$ so that the sequence $\Pi(\widetilde{\phi_{\alpha}^{-\varepsilon_n}}\cdot\tphi)$ does not have any discriminant for all $n\in\N$. Since $\widetilde{\phi_{\alpha}^{-\varepsilon_n}}\cdot\tphi\cleq\tphi$ and the sequence $(\widetilde{\phi_{\alpha}^{-\varepsilon_n}}\cdot\tphi)$ converges to $\tphi$ in the $C^1$-topology we conclude using the previous case and the monotonicity of the $\alpha$-selector. 
\end{proof}

\section{Spectral selectors for lens spaces}\label{sec : spectral selectors for lens spaces}

In \cite{allais2024spectral} the authors construct  a $\alpha_k$-spectral selector $c_k: \widetilde{\conto}(L_k^{2n+1},\ker\alpha_k)\to\R$ for any $n\in\N_{>0}$ and $k\geq 2$ (denoted in \cite{allais2024spectral} by $c_0$ or $c_+$). Its construction, that we present briefly below, is based on generating functions techniques and Givental's non-linear Maslov index \cite{granja2014givental,Giv90}. From now on we fix $n\in \N_{>0}$ and an integer $k\geq 2$.

 \subsubsection{Givental's non-linear Maslov index}  The Givental's non-linear Maslov index $\mu :\widetilde{\conto}(L_k^{2n+1},\xi_k)\to\Z$ is a quasimorphism that detects discriminant points in the following sense: if $\{\tphi_t\}_{t\in[0,1]}\subset\widetilde{\conto}(L_k^{2n+1},\xi_k)$ is a smooth path, i.e. $(\phi_t:=\Pi(\tphi_t))$ is a contact isotopy, such that $\mu(\tphi_0)\ne\mu(\tphi_1)$, then there exists $t_0\in [0,1]$ such that $\Pi(\tphi_{t_0})$ has a discriminant point. Denoting by $\widetilde{\phi_{\alpha_k}^T}:=[(\phi_{\alpha_k}^{tT})]\in\widetilde{\conto}(L_k^{2n+1},\xi_k)$, for $T\in\R$, it is shown moreover in \cite{Giv90,granja2014givental,allais2024spectral} that $\mu$ has the following properties.
 \begin{lem}[\cite{Giv90,granja2014givental,allais2024spectral}]\label{lem : maslov index}
 For all $\tphi,\tpsi\in\widetilde{\conto}(L_k^{2n+1},\xi_k)$\
 \begin{enumerate}
 \item the map $\R\to\Z$, $T\mapsto \mu(\widetilde{\phi_{\alpha_k}^{-T}}\cdot\tphi)$ is lower semi-continuous
 \item if $\mu(\tphi)$ is even then $\mu(\tpsi\cdot\tphi)\leq \mu(\tpsi)+\mu(\tphi)$
 \item $\mu(\widetilde{\phi_{\alpha_k}^T})=2(n+1)\left\lceil\frac{T}{2\pi}\right\rceil$ for all $T\in\R$
 \item $\mu(\tphi)\leq\mu(\tpsi)$ if $\tphi\cleq\tpsi$,
 \end{enumerate}
 where $\cleq$ is the bi-invariant binary relation introduced by Eliashberg-Polterovich in \cite{EP00}, i.e. $\tphi\cleq\tpsi$ if $\tpsi\cdot\tphi^{-1}$ can be represented by a contact isotopy $(\phi_t)$ such that $\alpha_k\left(\frac{d}{dt}\phi_t\right)\geq 0$. 
 \end{lem}
 Let us give a brief overview of its construction. To do consider $\tphi\in\widetilde{\conto}(L_k^{2n+1},\xi_k)$ and $(\phi_t)$ any contact isotopy representing it.\\

 Firstly, let us denote by $(\overline{\phi_t})\subset\conto(\sphere{2n+1},\xi_1)$ the lifted $\Z_k$-equivariant contact isotopy starting at the identity associated to the contact isotopy $(\phi_t)$, i.e. $\pi_1^k\circ\overline{\phi_t}=\phi_t\circ\pi_1^k$ for all $t\in[0,1]$ and $\overline{\phi_0}=\id$ where $\pi_1^k : \sphere{2n+1}\to L_k^{2n+1}$ is the covering map. We want to construct a path of $C^1$-functions $(f_t : L_k^{2(n+N)+1}\to\R)$, for some integer $N\geq 0$, so that for all $t\in[0,1]$
 \[\{x\in L_k^{2(n+N)+1}\ |\ \ud_x f_t=0,\ f_t(x)=0\}\overset{1-1}\leftrightarrow \{\pi_1^k(x)\in L_k^{2n+1}\ |\ x \text{ is a disc. point of } \overline{\phi}_t\},\]
 i.e. there is a one-to-one correspondance between the critical points of $f_t$ with critical value 0 and the discriminant points of $\phi_t$ that lift to a $\Z_k$-family of discriminant points of $\overline{\phi}_t$.
 %To do so, for a manifold $X$ and an integer $N$, consider the trivial vector bundle $X\times\R^N$ and denote by $\mathbb{O}_{\R^N}\subset T^*\R^N$ the $0$-section of the cotangent of $\R^N$. A function $F: X\times\R^N\to\R$ is said to be a generating function if the image of $\ud F : X\times\R^N\to T^*(X\times\R^N)=T^*X\times T^*\R^N$ intersects transversally $T^*X\times \mathbb{O}_{\R^N}$. In this case $\Sigma_F :=\ud F^{-1}(T^*X\times \mathbb{O}_{\R^N})$ is a submanifold and 
 %\[i_F : \Sigma_F \to T^*X\quad\quad\quad (x,v)\mapsto (x,\partial_x F(x,v))\]
 %is an immersion. For $X=\R^{2n}$ we say that $F: \R^{2n}\times\R^N\to\R$ is a generating function for a diffeomorphism\footnote{it has to be a symplectomorphism} $\phi$ of $\R^{2n}$ if $F$ is a generating function and $i_F : \Sigma_F\hookrightarrow T^*\R^{2n}$ is an embedding with image $\tau(\text{graph}(\phi))$ where $\text{graph}(\phi)=\{(x,\phi(x))\ |\ x\in \R^{2n}\}$ and 
 %\[\tau : \R^{2n}\times\R^{2n}\to T^*\R^{2n} \quad \quad (z,Z)\mapsto \left(\frac{z+Z}{2},i(z-Z)\right).\] 
 This path of function $(f_t)$ is constructed in the following way. Identifying the symplectization of $(\sphere{2n+1},\xi_1)$ with $(\R^{2n+2}\setminus\{0\},\omega_0)$, the $\Z_k$-equivariant contact isotopy $(\overline{\phi}_t)\subset\conto(\sphere{2n+1},\xi_1)$, can itself be uniquely lifted to a $\Z_k\times \R_{>0}$-equivariant Hamiltonian isotopy $(\psi_t)\subset\ham(\R^{2n+2}\setminus\{0\},\omega_0)$ 
 \[\psi_t(z)=e^{-\frac{1}{2}g_t(z/\|z\|)}\|z\|\overline{\phi}_t\left(\frac{z}{\|z\|}\right),\quad \text{ for all } t\in[0,1],\ z\in\R^{2n+2}\setminus\{0\},\]
 where $g_t : \sphere{2n+1}\to\R$ is the $\alpha_1$-conformal factor of $\overline{\phi}_t$, i.e. $\overline{\phi}_t^*\alpha_1=e^{g_t}\alpha_1$. Extending $\psi_t$ to the whole $\R^{2n+2}$ by putting $\psi_t(0)=0$ the authors in \cite[Prop 2.2 and Prop 2.14]{granja2014givental} show the existence of an integer $m\in\N$ and a path of $C^1$-functions $(F_t :\R^{2n+2}\times\R^{2N}\to\R)$, where $2N=(2n+2)\cdot 2m$, such that denoting by $\Sigma_{F_t}:=\{(x,v)\in\R^{2n+2}\times\R^{2N}\ |\ \partial_v F(x,v)\equiv 0\in T^*\R^{2n}\}$ the fiber critical points of $F_t$ we have\footnote{ in the following $\R^{2j}$ is identified with $\C^{j}$, for all $j\in\N_{>0}$, as in the introduction and $T^*\R^{2n+2}$ with $T\R^{2n+2}\simeq\R^{2n+2}\times\R^{2n+2}$ using the standard Euclidean metric}
 \begin{enumerate}
    % \item $F_t$ is smooth in a neighborhood $\Sigma_{F_t}\setminus\{0\}$ where $\mathbb{O}_{\R^{2N}}$ denotes the $0$-section of $T^*\R^{2N}$ 

         \item $i_{F_t} :\Sigma_{F_t}\to T^*\R^{2n+2},\ (x,v)\hookrightarrow \left(x,\partial_x F_t(x,v)\right)$ is injective and moreover 
         \[i_{F_t}(\Sigma_{F_t})=\tau(\mathrm{graph}(\psi_t))\]
          where $\text{graph}(\psi_t):=\{(x,\psi_t(x))\ |\ x\in\R^{2n+2}\}$ and $\tau(z,Z)=\left(\frac{z+Z}{2},i(z-Z)\right)$ for all $z,Z\in\R^{2n+2}\times\R^{2n+2}.$
          %\simeq\C^{n+1}\times\C^{n+1}.
 
     \item $F_t(\lambda e^{\frac{2i\pi j}{k}}z)=\lambda^2 F_t(z)$ for all $\lambda>0$, $z\in\R^{2n+2}\times\R^{2N}$ and $j\in\N$
     %\C^{n+1}\times\C^{N}$ 
     \item $F_0 :\R^{2n+2}\times(\R^{2n+2})^{2m}\to\R$, $(z,(Z_1,\cdots,Z_{2m}))\mapsto -2\sum\limits_{j=1}^m\left<Z_{2j}-z,i(Z_{2j-1}-z)\right>$.
      \end{enumerate}
 The path of functions $(F_t)$ induces the desired  path $(f_t : L_k^{2(n+N)+1}\to\R)$ and $(F_t)$ is called a based family of generating functions associated to $(\phi_t)$. \\

% There exists a $\Z_k\times\R_{>0}$ equivariant homeomorphism $\phi$ of $\R^{2n+2}\times\R^{2N}$ so that $\phi$ sends the fiber $\{x\}\times\R^{2N}$ to itself and $\phi|_{\Sigma_{F_0}}\setminus\{0\}$ is a diffeomorphism so that $F_0\circ\phi(x,\cdot) : \R^{2N}\to\R$ is a non degenerate quadratic form for all $x\in\R^{2n+2}$.

 Secondly, suppose that $k=p\in\N_{>0}$ is a prime number. If $(\phi'_t)$ is another contact isotopy in $\conto(L_p^{2n+1},\xi_p)$ representing the same element $\tphi$ and if $(f_t : L_p^{2N+1}\to\R)$ and $(f_t': L_p^{2N'+1}\to\R)$ are two paths of functions that are induced by two based families of generating functions associated respectively to $(\phi_t)$ and $(\phi'_t)$, then it is proven in \cite[Prop. 2.20 and Section 4]{granja2014givental} that
 \[\ind(\{f_0\leq 0\})-\ind(\{f_1\leq 0\})=\ind(\{f'_0\leq 0\})-\ind(\{f_1'\leq 0\}).\]
 Here the index $\ind(A)$ of an open subset $A\subset L_p^{j}$, for any integer $j\geq 3$, is the dimension over the field $\Z_p$ of the image of $i_A ^* : H^*(L_p^{j},\Z_p)\to H^*(A,\Z_p)$, where $i_A$ denotes the inclusion and $i_A^*$ the induced map on singular homology,  and the index of a closed subset $F$ is by definition $\ind(F):=\min\{\ind(A)\ |\ A \text{ open and } F\subset A\}$. The Givental's non-linear Maslov index of $\tphi$ is then defined to be
 \[\mu(\tphi):=\ind(\{f_0\leq 0\})-\ind(\{f_1\leq 0\}).\]
 
 Finally, for a general integer $k\geq 2$, denoting by $p\in \N_{>0}$ the smallest prime dividing $k$, we define $\mu(\tphi)$ by
 \[\mu(\tphi):=\mu(\Leg_k^p(\tphi)),\]
 where $\Leg_k^p : \widetilde{\conto}(L_k^{2n+1},\xi_k)\to\widetilde{\conto}(L_p^{2n+1},\xi_p)$ denotes the natural lift map. 
 %for any element $\tphi\in\widetilde{\conto}(L_k^{2n+1},\xi_k)$ we define $\mu(\tphi)$ to $\tpsi\in\widetilde{\conto}(L_p^{2n+1},\xi_p)$ where $p$ is the smallest prime number dividing $k$ and define $\mu(\tphi):=\mu(\tpsi)$.

\subsubsection{The spectral selector on $(L_k^{2n+1},\xi_k)$} The $\alpha_k$-spectral introduced in \cite{allais2024spectral} is defined by
\[c_k :\widetilde{\conto}(L_k^{2n+1},\xi_k)\to\R\quad \quad \tphi\mapsto \inf\left\{T\in\R\ |\ \mu\left(\widetilde{\phi_{\alpha_k}^{-T}}\cdot\tphi\right)\leq 0\right\}.\]
Before stating its main properties we introduce some notations. For any $\tphi\in\widetilde{\conto}(L_k^{2n+1},\xi_k)$ we consider the following set
\[\overline{\spec^{\alpha_k}}(\tphi):=\spec^{\alpha_1}(\overline{\tphi})\subset\spec^{\alpha_k}(\tphi)\]
where $\overline{\tphi}\in\widetilde{\conto}(\sphere{2n+1},\xi_1)$ denotes the $\Z_k$-equivariant lift of $\tphi$. For any positive number $T>0$ and any $x\in\R$ we denote by $\lceil x \rceil_{T}:=T\left\lceil \frac{x}{T}\right\rceil$ the smallest multiple of $T$ greater or equal to $x$.

%Note that the 3rd and 4th property of Lemma \ref{lem : maslov index} implies that $c_k(\tphi)$ is well defined and the lower semi-continuity property that 
%\begin{equation}\label{eq : lower-semicontinuous}
%c_k(\tphi)=\min\left\{T\in\R\ |\ \mu\left(\widetilde{\phi_{\alpha_k}^{-T}}\cdot\tphi\right)\leq 0\right\}.
%\end{equation}
 
% for $(\psi_t)$ such that $F_t$ is smooth in a neighborhood of $\Sigma_F\setminus\{0\}$ and such that $F_t(re^{\frac{2i\pi m}{k}}(z_1,\cdots,z_n),(z_1,\cdots,z_N))=r^2 F_t((z_1,\cdots,z_n),(z_1,\cdots,z_N))$. This path of functions $(F_t)$ therefore induces the desired path $(f_t : L_k^{2(n+N)}\to\R)$. 

 %When there is no ambiguity we will often write $c$ instead of $c_k$. 

\begin{prop}[\cite{allais2024spectral}]\label{prop : spectral selector for L_k^{2n+1}} The map $c_k : \widetilde{\conto}(L_k^{2n+1},\xi_k)\to\R$ is a $\alpha_k$-spectral selector that moreover satisfies for all $\tphi,\tpsi\in\widetilde{\conto}(L_k^{2n+1},\xi_k)$
\begin{enumerate}  
\item $c_k(\id)=0$ 
\item $\tphi\cleq\tpsi$ implies that $c_k(\tphi)\leq c_k(\tpsi)$
\item $\id\cleq\tphi$ and $\tphi\ne\id$ implies that $c_k(\tphi)>0$
\item $c_k(\tpsi\cdot\tphi)\leq c_k(\tpsi)+\lceil c_k(\tphi)\rceil_{\frac{2\pi}{k}}$
%\item $\lceil c_k(\tphi)\rceil_{\frac{2\pi}{k}}=-\lfloor c_k(\tphi^{-1})\rfloor_{\frac{2\pi}{k}}$
%\item $0\leq\lceil c_k(\tpsi)\rceil_{\frac{2\pi}{k}}+\lfloor c_k(\tpsi^{-1})\rfloor_{\frac{2\pi}{k}}\leq 2\pi+\frac{2\pi}{k}$. 

\item for any contact isotopy $(\phi_t)$ representing $\tphi$ \[\int_0^1\min \alpha_k\left(\frac{d}{dt}\phi_t\right)\ \ud t\leq c_k(\tphi)\leq\int_0^1\max \alpha_k\left(\frac{d}{dt}\phi_t\right)\ \ud t\]
\item $c_k(\tphi)\in\overline{\spec^{\alpha_k}}(\tphi)$
\item $\lceil c_k(\tpsi\cdot\tphi\cdot\tpsi^{-1})\rceil _{\frac{2\pi}{k}}= \lceil c_k(\tphi) \rceil _{\frac{2\pi}{k}}$. 
\end{enumerate}
\end{prop}

%If $(\phi_t)$ is a contact isotopy starting at the identity in a closed contact manifold $(M,\ker\alpha)$ we call its $\alpha$-Hamiltonian function, the only function $h:[0,1]\times M\to\R$ satisfying \begin{equation}\label{eq : fnct ham}
%    \alpha\left(\frac{d}{dt} \phi_t(x)\right)=h_t(\phi_t(x)).
% \end{equation}   
%    In Section \ref{sec : computation of the capacity} we recall that one can go in the converse direction : fixing a function $h:[0,1]\times M\to\R$ one can associate an only contact isotopy $(\phi_t)$ starting at the identity such that \eqref{eq : fnct ham} is satisfied. The following Proposition is also proved in \cite[Thm 1.1]{allais2024spectral}.

%\begin{prop}[\cite{allais2024spectral}]\label{lem : comparaison hamiltonienne}
%    Let $h : [0,1]\times L_k^{2n+1}\to\R$ be an $\alpha_k$-Hamiltonian function generating a contact isotopy $(\phi_t)$ then \[\int_0^1\min h_t\ \ud t\leq c_k([(\phi_t)])\leq\int_0^1\max h_t\ \ud t.\]
%\end{prop}

Before concluding this section it is worth noticing that we have a stronger triangular inequality and invariance by conjugation for elements in \[\widetilde{\conto}^{\mathbb{T}^1}(L_k^{2n+1},\xi_k):=\left\{ \tpsi\in\widetilde{\conto}(L_k^{2n+1},\xi_k)\ \left|\
        \parbox{6cm}{there exists $(\psi_t)$ representing $\tpsi$
    such that $\phi_{\alpha_k}^s\circ\psi_t=\psi_t\circ\phi_{\alpha_k}^s$ for all $s\in\R$ and $t\in[0,1]$}\right.\right\}.\]
%if $\tpsi\in\widetilde{\conto}(L_k^{2n+1},\ker\alpha_k)$ is $\mathbb{T}^1$-equivariant, i.e. $\widetilde{\phi_{\alpha_k}^t}\cdot\tpsi=\tpsi\cdot\widetilde{\phi_{\alpha_k}^{-t}}$ for all $t\in\R$, then there is a stronger triangle inequality satisfied by the selector :

\begin{prop}
    \label{lem : inegalite triangulaire}
   For any $\tpsi\in \widetilde{\conto}^{\mathbb{T}^1}(L_k^{2n+1},\xi_k)$, $\tphi\in\widetilde{\conto}(L_k^{2n+1},\xi_k)$ %\begin{enumerate}
   \[c_k(\tpsi\cdot\tphi)\leq c_k(\tpsi)+c_k(\tphi)\]
 % \item $c_k(\tpsi\cdot\tphi\cdot\tpsi^{-1})=c_k(\tphi)$.
 % \end{enumerate}
\end{prop}
%\begin{rem}
%    For the purpose of this paper the second statement of Proposition \ref{lem : inegalite triangulaire} is not relevant. 
%\end{rem}

To prove Proposition \ref{lem : inegalite triangulaire}, let us recall that $\phi\in\conto(L_k^{2n+1},\xi_k)$ is said to be $\alpha_k$-nondegenerate if for all $\alpha_k$-translated points $x\in L_k^{2n+1}$ of $\phi$ there does not exist any $X\in T_x L_k^{2n+1}\setminus \{0\}$ and $T\in\R$ such that
 \[\ud_x(\phi_{\alpha_k}^{-T}\circ\phi) (X)=X \text{ and } \ud_x g(X)=0\]
 where $g : L_k^{2n+1}\to\R$ is the $\alpha_k$-conformal factor of $\phi$, i.e. $\phi^*\alpha_k=e^g\alpha_k$. The next Lemma is also contained in \cite[Thm 1.1]{allais2024spectral}.

 \begin{lem}\label{lem : non-deg}
   Let $\tphi\in\widetilde{\conto}(L_k^{2n+1},\xi_k)$. If $\Pi(\tphi)$ is $\alpha_k$-nondegenerate then $\mu(\widetilde{\phi_{\alpha_k}^{-T}}\cdot\tphi)=0$ where $T:=c_k(\tphi)$.
 \end{lem}

%To prove the 1st point of Proposition \ref{lem : inegalite triangulaire} we will use the following Lemma.

%\begin{lem}\label{lem : generique}
%    For any $\tpsi\in\widetilde{\conto}^{\mathbb{T}^1}(L_k^{2n+1},\xi_k)$ there exists a sequence $(\tpsi_m)\subset \widetilde{\conto}(L_k^{2n+1},\xi_k)$ converging to $\tpsi$ in the $C^1$-topology so that for all $m\in\N$
 %   \begin{enumerate}
  %  \item $\tpsi_m \in\widetilde{\conto}^{\mathbb{T}^1}(L_k^{2n+1},\xi_k)$
 %   \item $\Pi(\tpsi_m)$ has only finitely many translated points $x_0,\cdots,x_{n_m}$ for some $n_m\in\N$.
  %  \item if $x$ and $x'$ are two different translated points of $\Pi(\tpsi_m)$ such that $\Pi(\tpsi_m)(x)=\phi_{\alpha_k}^{T}(x)$ and $\Pi(\tpsi_m)(x')=\phi_{\alpha_k}^{T'}(x')$ then    
%\[\left\{T +\frac{2\pi}{k}j\ |\ j\in\Z\right\}\ne\left\{T'+\frac{2\pi}{k}k\ |\ j\in\Z\right\}.\]
 %   \end{enumerate}
%\end{lem}

%We prove Lemma \ref{lem : generique} in Section \ref{sec : computation of the capacity}. 

\begin{proof}[Proof of Proposition] \ref{lem : inegalite triangulaire}\
   Let us first suppose that $\Pi(\tphi)$ is $\alpha_k$-nondegenerate and denote by $T:=c_k(\tphi)$, $T':=c_k(\tpsi)$. Then by Lemma \ref{lem : non-deg} we have that $\mu(\widetilde{\phi_{\alpha_k}^{-T}}\cdot\tphi)=0$ is even.  Therefore, since $\tpsi$ commutes with $\widetilde{\phi_{\alpha_k}^s}$ for all $s\in\R$ we have
   \[\begin{aligned}\mu(\widetilde{\phi_{\alpha_k}^{-(T'+T)}}\cdot\tpsi\cdot\tphi)=\mu(\widetilde{\phi_{\alpha_k}^{-T'}}\cdot\tpsi\cdot\widetilde{\phi_{\alpha_k}^{-T}}\cdot\tphi)&\leq \mu(\widetilde{\phi_{\alpha_k}^{-T'}}\cdot\tpsi)+\mu(\widetilde{\phi_{\alpha_k}^{-T}}\cdot\tphi)\\
   &\leq \mu(\widetilde{\phi_{\alpha_k}^{-T'}}\cdot\tphi)\leq 0,
   \end{aligned}\]
   where the first inequality follows from the 2nd point of Lemma \ref{lem : maslov index} and the last inequality from the 1st property of Lemma \ref{lem : maslov index}. This implies that \[c_k(\tpsi\cdot\tphi)\leq c_k(\tpsi)+c_k(\tphi)\] which finishes the proof in the case where $\Pi(\tphi)$ is non degenerate.
   
   \noindent
   The general case can now be deduced by the $C^1$-continuity of $c_k$ and the $C^1$-density, inside $\widetilde{\conto}(L_k^{2n+1},\xi_k)$, of the set of elements $\tpsi\in\widetilde{\conto}(L_k^{2n+1},\xi_k)$ such that $\Pi(\tpsi)$ is $\alpha_k$-nondegenerate 
  (it is a direct consequence of \cite[Lemma 3.1]{allais2024spectral} and the fact that $\Pi : \widetilde{\conto}(L_k^{2n+1},\xi_k)\to\conto(L_k^{2n+1},\xi_k)$ is a local homeomorphism when endowing both spaces with the $C^1$-topology).\end{proof}

      % \item Let $(\psi_t)$ be a contact isotopy representing $\tpsi$ such that $\phi_{\alpha_k}^s\circ\psi_t=\psi_t\circ\phi_{\alpha_k}^s$ for all $t\in[0,1]$ and $s\in\R$. This implies that $\psi_T^*\alpha_k=\alpha_k$, for all $T\in[0,1]$, and that 
    %   \[\spec^{\alpha_k}(\tpsi_T\cdot\tphi\cdot\tpsi_T^{-1})=\spec^{\alpha_k}(\tphi)\]
  %     where $\tpsi_T:=[(\psi_{tT})]$. Moreover it is proved in \cite[Lemma 3.2]{allais2024spectral} that $\spec^{\alpha_k}(\tphi)$ is nowhere dense which implies that the map $[0,1]\to\R,\ T\mapsto c_k(\tpsi_T\cdot\tphi\cdot\tpsi_T^{-1})\in\spec^{\alpha_k}(\tphi)$ is constant, because continuous, and thus
 %      \[c_k(\tphi)=c_k(\tpsi_0\cdot\tphi\cdot\tpsi_0^{-1})=c_k(\tpsi_1\cdot\tphi\cdot\tpsi_1^{-1})=c_k(\tpsi\cdot\tphi\cdot\tpsi^{-1}).\]
  % \end{enumerate}

\begin{rem}
A $\mathbb{T}^1$-equivariant contactomorphism $\phi\in\conto(L_k^{2n+1},\xi_k)$ is always $\alpha_k$-degenerate since its $\alpha_k$-conformal factor is constantly equal to $0$, i.e. $\phi^*\alpha_k=\alpha_k$, and $(\phi_{\alpha_k}^{-T}\circ\phi)^*R_{\alpha_k}=R_{\alpha_k}$ for all $T\in\R$.  
%for any $\alpha_k$-translated point $x$ of $\phi$ with translation $T\in\R$ we have
%\[(\phi_{\alpha_k}^{-T}\circ\phi)_* R_{\alpha_k}(x)=R_{\alpha_k}(x).\]
    %If $\phi\in\conto(L_k^{2n+1},\xi_k)$ is a $\mathbb{T}^1$-equivariant contactomorphism, i.e. $\phi\circ\phi_{\alpha_k}^s=\phi_{\alpha_k}^s\circ\phi$ for all $s\in\R$, then $\phi^*\alpha_k=\alpha_k$. Therefore $\phi$ is $\alpha_k$-degenerate since for any $\alpha_k$-translated point $x\in L_k^{2n+1}$ of $\phi$ with translation $T\in\R$ .
\end{rem}

\section{Contact capacity in Lens spaces}\label{sec : Contact capacity}

For any non-empty open set $U\subset L_k^{2n+1}$ we denote by $\widetilde{\contoc}(U,\xi_k)$ the subgroup of $\widetilde{\conto}(L_k^{2n+1},\xi_k)$ whose elements can be represented by contact isotopies $(\phi_t)$ supported in $U$, i.e. $\overline{\underset{t\in[0,1]}\bigcup\mathrm{Supp}(\phi_t)}\subset U$ where $\mathrm{Supp}(\phi_t):=\overline{\{x\in M\ |\ \phi_t(x)\ne x\}}$. Let $C_k :\{\text{non-empty open sets of } L_k^{2n+1}\}\to\R_{\geq 0}\cup\{+\infty\}$ be the map: 
\[U\mapsto C_k(U):=\sup\{ c_k(\tphi)\ |\ \tphi\in\widetilde{\contoc}(U,\xi_k)\}.\]

\begin{thm}\label{thm : contact capacity}\
For all non-empty open sets $U, V\subset L_k^{2n+1}$ and $\phi\in\conto(L_k^{2n+1},\xi_k)$
   \begin{enumerate}
   \item(monotonicity) $C_k(U)\leq C_k(V)$ if $U\subset V$
       \item(invariance by conjugation) $\lceil C_k(U)\rceil_{\frac{2\pi}{k}}=\lceil C_k(\phi(U))\rceil_{\frac{2\pi}{k}}$ 
       \end{enumerate}
     \begin{enumerate}
       \item[(3)](lower bound) if $\Psi : (\mathbb{B}^{2n}(R),\omega_0)\hookrightarrow (\CP^n,\omega)$ is a symplectic embedding for some $R>0$ then  $C_k(B_k(r))\geq \pi r^2$ for all $r\in(0, R]$ where $B_k(r):=\pi_k^{-1}(\Psi(\mathbb{B}^{2n}(r)))$,
       \item[(4)](upper bound) if $\Psi : (\mathbb{B}^{2n}(\lambda R),\omega_0)\hookrightarrow (\CP^n,\omega)$ is a symplectic embedding for some $R>0$ and $\lambda>\sqrt{3}$ then  $C_k(B_k(r))\leq\pi r^2$ for all $r\in(0,R]$ where $B_k(r):=\pi_k^{-1}(\Psi(\mathbb{B}^{2n}(r)))$.
    \end{enumerate}
\end{thm}
The monotonicity property of $C_k$ in Theorem \ref{thm : contact capacity} follows by definition. The invariance by conjugation of $C_k$ follows from the invariance by conjugation of the map $\widetilde{\conto}(L_k^{2n+1},\xi_k)\to\R,\ \tphi \mapsto \lceil c_k(\tphi)\rceil_{\frac{2\pi}{k}}$ (see Lemma \ref{lem : prequantization}). To give the complete proof of Theorem \ref{thm : contact capacity} it thus remains to prove the last two estimates. Let us admit them for the moment and deduce the proof of Theorem \ref{thm : non-squeezing optimal}:

\begin{proof}[Proof of Theorem \ref{thm : non-squeezing optimal}]
Suppose by contradiction that there exists $\phi\in\conto(L_k^{2n+1},\xi_k)$ such that $\phi(B^1_k(a_1))\subset B^2_k(a_2)$. Then,  using the 1st and 4th point of Theorem \ref{thm : contact capacity}, we have on one side \begin{equation}\label{eq : contradiction}
    C_k(\phi(B^1_k(a_1)))\leq C_k(B^2_k(a_2))\leq \pi a_2^2\leq\frac{2\pi}{k}j.
    \end{equation} But on the other side using the 2nd and the 3rd point we get that
    \[ \lceil C_k(\phi(B^1_k(a_1)))\rceil_{\frac{2\pi}{k}}=\lceil C_k(B^1_k(a_1))\rceil_{\frac{2\pi}{k}}\geq \left\lceil \pi a_1^2\right\rceil_{\frac{2\pi}{k}}>\frac{2\pi}{k}j\]
    which contradicts \eqref{eq : contradiction} and finishes the proof of Theorem \ref{thm : non-squeezing optimal}.
\end{proof}

We give the proof of the last two estimates of Theorem \ref{thm : contact capacity} in Section \ref{sec : computation of the capacity}. But before, in the next Section \ref{sec : capacite-energy} we prove a capacity-energy inequality satisfied by $C_k$ which will be crucial to prove the lower bound stated in Theorem \ref{thm : contact capacity}.

\section{Capacity-energy Inequality}\label{sec : capacite-energy}
%Before stating the capacity-energy inequality (see Proposition \ref{thm : capacite-energie}) in the case of lens spaces we start by giving some definitions and phenomena that can be studied in the more general context of any closed contact manifold $(M,\xi=\ker\alpha)$. \\

Let $(M,\ker\alpha)$ be a closed contact manifold. For any subset $U\subset M$ we denote by $\mathcal{O}_\alpha(U):=\underset{t\in\R}\bigcup\phi_\alpha^t(U)$ its $\alpha$-orbit.

\begin{definition}
 A diffeomorphism $\phi$ of $M$ $\alpha$-displaces a subset $U\subset M$ if $\phi$ displaces the $\alpha$-orbit of $U$, i.e.  
$\phi(\mathcal{O}_\alpha(U))\bigcap\mathcal{O}_\alpha(U)=\emptyset.$
\end{definition}

\begin{rem}
\begin{enumerate}
    \item If $(M,\xi)$ is a closed cooriented contact manifold, the Reeb flow of a generic contact form supporting $\xi$ has a dense orbit \cite[Theorem C.4]{colin2023generic}. Hence, for such contact form $\alpha$, a diffeomorphism $\phi$ cannot displace the closure of the $\alpha$-orbit of any non-empty open set $U$, i.e. $\phi(\overline{\mathcal{O}_\alpha(U)})\cap\overline{\mathcal{O}_\alpha(U)}\ne \emptyset$.
    \item In the unitary cotangent of the torus $\T^n\times\sphere{n-1}:=(\R^n/\Z^n)\times \sphere{n-1}$ endowed with its standard contact form $\alpha$, i.e. $\alpha_{(q,p)}(Q,P):=\left<p,Q\right>$ where $\left<\cdot,\cdot\right>$ denotes the standard Euclidean scalar product, the $\alpha$-Reeb flow does not admit any dense orbit, for any integer $n\geq 2$. However, for a generic $p\in \sphere{n-1}$, the $\alpha$-orbit of $(q,p)\in\T^n\times\sphere{n-1}$, for any $q\in\T^n$, is dense in $\T^n\times\{p\}$. Moreover $\T^n\times\{p\}$ is a Prelagrangian that cannot be displaced from itself by any contactomorphism that is isotopic to the identity \cite[Example 2.4.A]{EP00}. Therefore $\phi(\overline{\mathcal{O}_\alpha(U)})\cap \overline{\mathcal{O}_\alpha(U)}\ne\emptyset$ for any non-empty open set $U\subset \T^n\times \sphere{n-1}$ and any $\phi\in\conto(\T^n\times\sphere{n-1},\ker\alpha)$. 
    \item The first two remarks has to be put in contrast with the following fact: it has been shown in \cite{Sauglam2018ContactFW} that any contact manifold $(M,\xi)$ admits a contact form $\alpha$ and a non-empty open set $U\subset M$ where the $\alpha$-Reeb flow is $1$-periodic on $U$. 
\end{enumerate}
\end{rem}

The aim of this Section is to prove the following:

\begin{thm}\label{thm : capacite-energie}
   Let $U\subset L_k^{2n+1}$ and $\tpsi\in\widetilde{\conto}(L_k^{2n+1},\xi_k)$. If $\Pi(\tpsi)$ $\alpha_k$-displaces $U$ then
    $\lceil c_k(U)\rceil_{\frac{2\pi}{k}}\leq c_k(\tpsi^{-1})+\lceil c_k(\tpsi)\rceil_{\frac{2\pi}{k}}$. If moreover $\tpsi\in\widetilde{\conto}^{\mathbb{T}^1}(L_k^{2n+1},\xi_k)$ then $c_k(U)\leq c_k(\tpsi^{-1})+c_k(\tpsi).$
\end{thm}

\begin{rem}
One can use \cite{allais2024spectral}, (see for example the proof of \cite[Prop 5.3]{allais2024spectral} combined with the Poincaré duality of \cite[Thm 1.1]{allais2024spectral}), to show moreover that for every $\tpsi\in\widetilde{\conto}(L_k^{2n+1},\xi_k)$
\[0\leq\lceil c_k(\tpsi)\rceil_{\frac{2\pi}{k}}+\lfloor c_k(\tpsi^{-1})\rfloor_{\frac{2\pi}{k}}\leq 2\pi+\frac{2\pi}{k}.\] 
 where $\lfloor x\rfloor_{\frac{2\pi}{k}}:=\frac{2\pi}{k}\left\lfloor x\cdot\frac{k}{2\pi}\right\rfloor$ is the biggest multiple of $\frac{2\pi}{k}$ smaller or equal to $x$ for any $x\in\R$. In particular if $U\subset L_k^{2n+1}$ is $\alpha_k$-displaceable then $\lceil c_k(U)\rceil_{\frac{2\pi}{k}}\leq 2\pi +\frac{4\pi}{k}$.
\end{rem}

%Suppose from now on that $(M,\xi=\ker\alpha)$ carries a $\alpha$-spectral $c:\widetilde{\conto}(M,\xi)\to\R$ and that the $\alpha$-spectrum of any element of $\widetilde{\conto}(M,\xi)$. We then have the following

Theorem \ref{thm : capacite-energie} will be a consequence of the next Proposition that can be stated for a general closed contact manifold $(M,\ker\alpha)$ admitting a $\alpha$-spectral selector. For any non-empty open set $U\subset M$, as before, $\widetilde{\contoc}(U,\xi)\subset\widetilde{\conto}(M,\xi)$ denotes the subgroup of elements that can be represented by contact isotopies supported in $U$.

%that can be stated in the more general context of any closed contact manifold $(M,\ker\alpha)$ admitting a $\alpha$-spectral selector. 

\begin{prop}\label{prop : displacement}
    Let $U\subset M$ be a non-empty open set of a closed contact manifold $(M,\ker\alpha)$ and $\tphi\in\widetilde{\conto}^c(U,\xi)$. If $c :\widetilde{\conto}(M,\ker\alpha)\to\R$ is an $\alpha$-spectral selector and if $\tpsi\in\widetilde{\conto}(M,\xi)$ is such that $\Pi(\tpsi)$ $\alpha$-displaces $U\subset M$ then 
    \[c(\tpsi\cdot\tphi)=c(\tpsi).\]
    
\end{prop}

Proposition is a consequence of the following:

\begin{lem}\label{lem : spectre displacement}
Let $\tpsi,\tphi\in\conto(M,\xi)$. If $\Pi(\tpsi)$ $\alpha$-displaces the support of $\Pi(\tphi)$ then 
\[\spec^\alpha(\tpsi\cdot\tphi)\subset\spec^\alpha(\tpsi).\]
\end{lem}

\begin{proof}
Let $t\in\spec^\alpha(\tpsi\cdot\tphi)$ and $x\in M$ a discriminant point of $\phi_\alpha^{-t}\circ (\psi\circ\phi)$ where $\psi:=\Pi(\tpsi)$ and $\phi:=\Pi(\tphi)$. In particular $\psi(\phi(x))=\phi_\alpha^t(x)$. Suppose by contradiction that $x$ lies in the support of $\phi$. Then $\phi(x)\in\mathrm{Supp}(\phi)\subset\mathcal{O}_\alpha(\mathrm{Supp}(\phi))$ and $\phi_\alpha^t(x)\in\mathcal{O}_\alpha(\mathrm{Supp}(\phi))$. But this contradicts the fact that $\psi(\mathcal{O}_\alpha(\mathrm{Supp}(\phi))\cap\mathcal{O}_\alpha(\mathrm{Supp}(\phi))=\emptyset$. Therefore $x$ is not in the support of $\phi$. Thus $x$ is a discriminant point of $\phi_\alpha^{-t}\circ\psi$ and $t\in\spec^\alpha(\tpsi)$. In particular $\spec^\alpha(\tpsi\cdot\tphi)\subset\spec^\alpha(\tpsi)$. \end{proof}

\begin{proof}[Proof of Proposition \ref{prop : displacement}]
Let $(\varphi_t)_{t\in[0,1]}$ be a contact isotopy supported in $U$ representing $\tphi$ and $\tphi_s:=[(\varphi_{ts})]$ for all $s\in[0,1]$. The map $\gamma : [0,1]\to\R$, $s\mapsto c(\tpsi\cdot\tphi_s)$ is continuous (by Definition \ref{def : spectral selector}) and takes value in $\spec^\alpha(\tpsi)$ thanks to Lemma  \ref{lem : spectre displacement}. Since $\spec^\alpha(\tpsi)$ is nowhere dense (see \cite[Lemma 2.11]{albers}) we conclude that $\gamma$ is constant and in particular $\gamma(0)=c(\tpsi)=\gamma(1)=c(\tpsi\cdot\tphi)$. \end{proof}

We are now ready to prove Theorem \ref{thm : capacite-energie}:

\begin{proof}[Proof of Theorem \ref{thm : capacite-energie}] 
Consider $\tphi\in\widetilde{\contoc}(U,\xi_k)$ then 
    \begin{equation}\label{eq : displacement}
    \lceil c_k(\tphi)\rceil_{\frac{2\pi}{k}}=\lceil c_k(\tpsi^{-1}\cdot\tpsi\cdot\tphi)\rceil_{\frac{2\pi}{k}}\leq c_k(\tpsi^{-1})+\lceil c_k(\tpsi\cdot\tphi)\rceil_{\frac{2\pi}{k}}=c_k(\tpsi^{-1})+\lceil c_k(\tpsi)\rceil_{\frac{2\pi}{k}},
    \end{equation}
    where the last equality follows from Proposition \ref{prop : displacement}. This proves the first statement. Moreover under the assumption that $\tpsi\in\widetilde{\conto}^{\mathbb{T}^1}(L_k^{2n+1},\xi_k)$ one can use Proposition \ref{lem : inegalite triangulaire} to show that the relations in \eqref{eq : displacement} remain true without taking the $\frac{2\pi}{k}$-ceiling function which proves the second statement.
\end{proof}

%Note that if $U\subset L_k^{2n+1}$ can be displaced by an element $\tpsi\in\widetilde{\conto}(L_k^{2n+1})$ which is $\mathbb{T}^1$-equivariant then $U$ is $\alpha$-displaceable by $\tpsi$.

%Using the stronger triangular inequality of Proposition \ref{lem : inegalite triangulaire} we get in a similar way the next Proposition.

%\begin{prop}\label{prop : strong capacite-energie}
%     Suppose that $U\subset L_k^{2n+1}$ can be displaced by $\tphi\in\widetilde{\conto}(L_k^{2n+1},\xi_k)$ which is $\mathbb{T}^1$-equivariant. Then
 %   \[c_k(U)\leq c_k(\tphi^{-1})+ c_k(\tphi).\]
%\end{prop}

%We leave the proof of the above Proposition \ref{prop : strong capacite-energie} to the reader.

\section{Estimation of the contact capacity}\label{sec : computation of the capacity}
Let us recall our conventions to describe Hamiltonian dynamics. $(W,\Omega)$ being a symplectic manifold, to any compactly supported function $H : [0,1]\times W\to\R$ we associate the path of symplectic vector fields $(X_H^t)$ on $W$, i.e. the flow $(\psi_H^t)$ of which is an isotopy of symplectomorphisms, defined by 
\begin{equation}
\Omega(X_t,\cdot)=-\ud H(t,\cdot)\quad \text{ for all } t\in[0,1].
\end{equation}
 We say that $H$ is the Hamiltonian function that generates the isotopy $(\psi_H^t)$ and denote by $\hamc(W,\omega)$ the set of diffeomorphisms arising as time $1$-map of such isotopies. 

%We call such time-dependant vector fields, Hamiltonian vector fields and denote by $\hamc(W,\omega)$ the set of compactly supported Hamiltonian diffeomorphisms, i.e. diffeomorphisms arising as time $1$-map of any Hamiltonian vector field. 

Similarly if $(M,\xi=\ker\alpha)$ is a closed contact manifold, with a fixed contact form $\alpha$, to any function $h : [0,1]\times M\to\R$ one associates a path of contact vector fields $(Y_h^t)$, i.e. the flow $(\phi_h^t)$ of which is a contact isotopy, defined by the two equations 
\begin{equation}\label{eq : contact vector field}
\left\{
    \begin{array}{ll}
       \alpha(Y_h^t)=h(t,\cdot) \\
        \ud \alpha(Y_h^t,\cdot)=\ud h^t(R_\alpha)\alpha-\ud h^t
    \end{array}
\right.
\end{equation}
where $R_\alpha$ denotes the Reeb vector field associated to $\alpha$ (the vector field defined by the above equations \eqref{eq : contact vector field} for the function constantly equal to $1$). We say that $h$ is the $\alpha$-Hamiltonian function that generates the contact isotopy $(\phi_h^t)$.

%, then Lemma \ref{lem : contact lift} holds replacing $(\CP^n,\omega)$ and $(L_k^{2n+1},\ker\alpha_k)$ by $(W,\omega)$ and $(M,\ker\alpha)$ respectively.  

We leave the proof of the following Lemma to the reader. 

\begin{lem}\label{lem : contact lift}\
  Let $T>0$ be any positive real number and $\pi$ be the projection of a $S^1:=\R/T$ principal bundle $M$ over a symplectic manifold $(W,\Omega)$ so that $\pi^*\Omega=\ud\alpha$ for some contact form $\alpha$ on $M$, whose Reeb flow induces the $S^1$-action.
  %  \begin{enumerate}
  %  \item Let $(W_1,\Omega_1)$, $(W_2,\Omega_2)$ be two symplectic manifolds such that $W_1$ is open. If there exists $\Psi : (W_1,\Omega_1)\hookrightarrow (W_2,\Omega_2)$ a symplectic embedding, then for any compactly supported function $H_1 : [0,1]\times W_1\to\R$ generating the isotopy $(\psi_1^t)$ the function
 %   \[H_2 : [0,1]\times W_2\to\R \quad  \quad 
 %   x\mapsto \left\{
  %  \begin{array}{ll}
  %     H_1(t,\Psi^{-1}(x)) & \text{ if } x\in \Psi(W_1)\\
  %      0 & \text{ otherwise}
  %  \end{array}
%\right. \]
%is a smooth compactly supported function that generates an isotopy $(\psi_2^t)\subset\hamc(W_2,\Omega_2)$ satisfying
%\[[0,1]\times W_2\to W_2\quad\quad(t,x)\mapsto \left\{
%\begin{array}{ll}
%\psi_2^t(x)=\Psi\circ\psi_1^t\circ\Psi^{-1}(x) & \text{ if } x\in \Psi(W_1)\\
 %      x & \text{ otherwise}
 %   \end{array}.
 %   \right.\] 
 %   \item 
 Then for any function $ H : [0,1]\times W\to\R$  the $\alpha$-Hamiltonian function \[h : [0,1]\times M\to\R\quad \quad \quad (t,x)\mapsto H(t,\pi(x))\] generates a contact isotopy $(\phi_h^t)\subset\conto(M,\ker\alpha)$ which is $S^1$-equivariant, i.e. $\phi_h^t\circ\phi_{\alpha}^s=\phi_{\alpha}^s\circ\phi_h^t$ for all $s\in\R$, $t\in [0,1]$. Moreover
    $\pi\circ \phi_h^t=\psi_H^t\circ\pi$ for all $t\in[0,1]$, where $(\psi_H^t)\subset\ham(W,\Omega)$ is the isotopy generated by $H$.
    
    %In particular if $\psi_H^1$ displaces an open set $\mathcal{U}\subset \CP^n$, i.e. $\psi_1(U)\cap U=\emptyset$, then $\phi_1$ $\alpha$-displaces $U:=\pi_k^{-1}(\mathcal{U})\subset L_k^{2n+1}$. 
\end{lem}

\subsection{Estimation of the lower bound}\label{sec : computation lower bound}
The aim of this subsection is to prove the lower bound in Theorem \ref{thm : contact capacity} satisfied by $C_k$, for any integer $k\geq 2$ and $n\in\N_{>0}$, which we recall in the next Proposition.
\begin{prop}\label{prop : lower bounds}
    Let $\Psi : (\mathbb{B}^{2n}(R),\omega_0)\hookrightarrow (\CP^n,\omega)$ be a symplectic embedding where $ R>0$. Then $C_k(B_k(r))\geq\pi r^2$ for all $r\in (0,R]$
    where $B_k(r):=\pi_k^{-1}(\Psi(\mathbb{B}^{2n}(r)))$. 
\end{prop}

The proof of Proposition \ref{prop : lower bounds} will be an adaptation of the proof of \cite[Lemma 3, Chapter 3]{hoferzehnder} stating that the Hofer-Zehnder  capacity of $(\mathbb{B}^{2n}(r),\omega_0)$ is bounded from below by $\pi r^2$.

\begin{proof}\

   \textbf{Step 1 : Construction in $(\mathbb{B}^{2n}(r),\omega_0)$}
   
   Consider $0<\varepsilon<\pi r^2$ and a smooth function $f : \R_{\geq 0}\to\R_{\geq0}$ supported in $[0,r^2)$
   %, i.e. $\overline{\{x\in[0,+\infty)\ |\ f(x)\ne 0\}}\subset[0,r^2)$,
   that satisfies moreover :
   \begin{enumerate}
   \item $f(t)=\pi r^2-\varepsilon$ for $t$ in a non-empty neighborhood of $\{0\}$
       \item $-\pi <f'(t)\leq 0$ for all $t\in\R_{\geq 0}$
       \item $f'(t)=0$ if and only if $f(t)\in\{\pi r^2-\varepsilon,0\}$
   \end{enumerate}
  \begin{center} \includegraphics[scale=0.5]{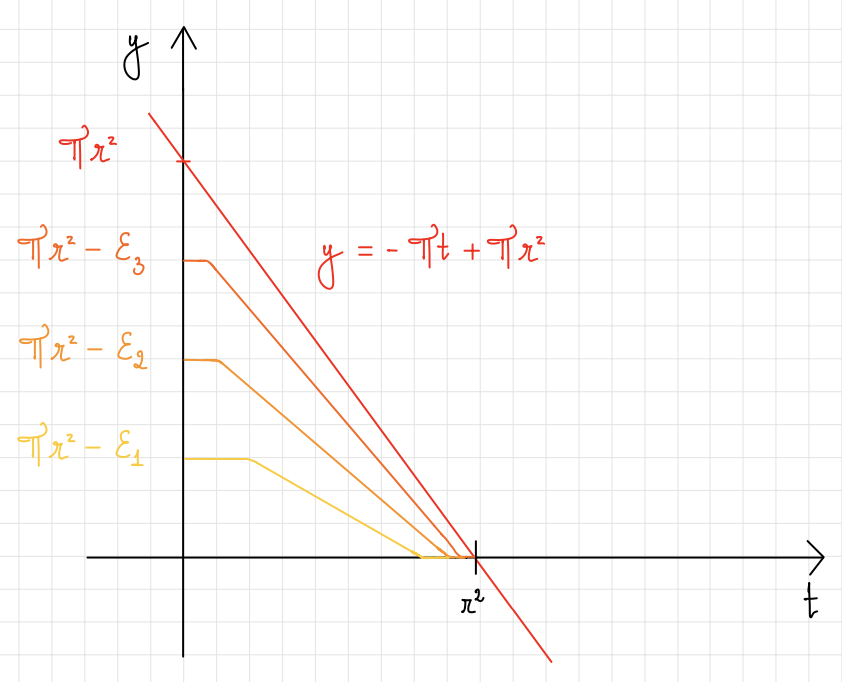}\\
  {\small{Figure 1: Graphs of different functions $f$ for different values of $\varepsilon$ with $0<\varepsilon_3<\varepsilon_2<\varepsilon_1$}}
\end{center}

Consider the open symplectic manifold $(\mathbb{B}^{2n}(r),\omega_0)$ and the compactly supported Hamiltonian function $H :\mathbb{B}^{2n}(r)\to\R,\ x\mapsto f(\|x\|^2)$, where $\|\cdot\|$ denotes the usual Euclidean norm of $\R^{2n}$. The vector field $X_H$ generated by $H$, i.e. $\omega_0(X_H,\cdot)=-\ud H$, satisfies for all $x\in\R^{2n}$
\[X_H(x)=J\nabla H(x)=2 f'(\|x\|^2) Jx\]
where $J$ is the standard complex structure of $\R^{2n}$\footnote{after identifying $\R^{2n}$ with $\C^{n}$ via $\Phi$ defined in the introduction $J:=\Phi^{-1}\circ i\circ\Phi$ corresponds to the multiplication by $i$} and $\nabla H$ denotes the gradient of $H$ with respect to the standard Euclidean scalar product. A direct computation shows that for all $t\in[0,1]$ and $x\in\mathbb{B}^{2n}(r)$ \begin{equation}\label{eq : equa diff}\psi_H^t(x)=e^{2f'(\|x\|^2)Jt}x=\cos(2f'(\|x\|^2)t)x+\sin(2f'(\|x\|^2)t)Jx,
\end{equation}
where $(\psi_H^t)\subset\hamc(\mathbb{B}^{2n}(r))$ denotes the flow generated by $H$.

%Indeed, putting
%$x(t):=\psi_H^t(x)$ we have $\dot x(t)=X_H(x(t))$ and
%\[\begin{aligned}\frac{d}{dt}\left<x(t),x(t)\right>=2\left<X_H(x(t)),x(t)\right>=-4f'(\|x(t)\|^2)\left<Jx(t),x(t)\right>=0.
%\end{aligned}\]
%Therefore $t\mapsto \|x(t)\|^2$ is constant and thus 
%\[\dot x(t)=2 f'(\|x(t)\|^2)Jx(t)=2f'(\|x\|^2)Jx(t) \]
%which proves the relation \eqref{eq : equa diff}. 

Note that since $2f'([0,+\infty[)\subset (-2\pi,0]$ and $f'(0)=0$, the orbit $t\mapsto\psi_H^t(x)$ of a point $x\in\mathbb{B}^{2n}(r)$ is $T$-periodic for some $T\in(0,1]$, i.e. $\psi_H^T(x)=x$, if and only if $f'(\|x\|^2)=0$. In other words the orbit of $x$ is $T$-periodic, with $T\in(0,1]$, if and only if $f(\|x\|^2)=H(x)\in\{\pi r^2-\varepsilon,0\}$ and $\ud_x H=0$.\\

\textbf{Step 2 : Import the construction to $(\CP^n,\omega)$}
Consider now on $(\CP^n,\omega)$ the Hamiltonian function 
\[G:\CP^n\to\R \quad  \quad 
    x\mapsto \left\{
    \begin{array}{ll}
       H(\Psi^{-1}(x)) & \text{ if } x\in \Psi(\mathbb{B}^{2n}(r))\\
        0 & \text{ otherwise}
    \end{array}.
\right. \]
A direct computation shows that the isotopy it generates $(\psi_G^t)\subset\ham(\CP^n,\omega)$ satisfies :
\[\psi_G^t : 
    x\mapsto \left\{
    \begin{array}{ll}
\Psi(\psi_H^t(\Psi^{-1}(x))) & \text{ if } x\in \Psi(\mathbb{B}^{2n}(r))\\
        x & \text{ otherwise}
    \end{array}.
\right. \]
In particular, thanks to the first step we get that the orbit $t\mapsto \psi_G^t(x)$ of a point $x\in\CP^n$ is $T$-periodic, with $T\in(0,1]$, if and only if $G(x)\in\{\pi r^2-\varepsilon,0\}$ and $\ud_x G=0$.\\

\textbf{Step 3 : Import the construction to $(L_k^{2n+1},\xi_k)$}
 Consider on $(L_k^{2n+1},\ker\alpha_k)$ the $\alpha_k$-Hamiltonian function 
\[h : L_k^{2n+1}\to\R,\quad \quad x\mapsto G(\pi_k(x)),\]
and denote by $(\phi_h^t)\subset\conto(L_k^{2n+1},\xi_k)$ the flow of contactomorphisms it generates. It is clear that its lift $(\overline{\phi_h^t})\subset\conto(\sphere{2n+1},\ker\alpha)$, starting at the identity, is generated by the $\alpha$-contact Hamiltonian $\overline{h} :\sphere{2n+1}\to\R$, $x\mapsto G(\pi_1(x))$, where for simplicity  $\alpha:=\alpha_1$. We claim that: if $x\in \sphere{2n+1}$ is a $\alpha$-translated point of $\overline{\phi_h^T}$ for some $T\in(0,1]$ then
    \begin{equation}\label{eq : translated point equadiff}
\text{there is }a\in\{\pi r^2-\varepsilon,0\} \text{ so that } \overline{\phi_h^t}(x)=\phi_{\alpha}^{at}(x) \text{ for all } t\in[0,1].
    \end{equation}
To prove this claim, fix $x\in \sphere{2n+1}$ and $T\in(0,1]$ so that $x$ is a  $\alpha$-translated of $\overline{\phi_h^T}$. By Lemma \ref{lem : contact lift} $x$ is a $\alpha$-translated point of $\overline{\phi_h^T}$ implies that $\pi_1(x)$ is a fixed point of $\psi_G^T$. This implies, by the previous step, that $G(\pi_1(x))\in\{\pi r^2-\varepsilon,0\}$ and $\ud_{\pi_1(x)} G =0$.  Therefore $\ud_x \overline{h}=0$ and $\overline{h}(x)\in\{\pi r^2-\varepsilon,0\}$. Since $\overline{h}\circ\phi_\alpha^t=\overline{h}$ for all $t\in\R$ we also get  
   \begin{equation}\label{eq : 3 equations}
   \ud \overline{h}(R_\alpha)=0,\quad \ud_{\phi_\alpha^t(x)} \overline{h}=0,\quad \text{ and }\quad \overline{h}(\phi_\alpha^t(x))=\overline{h}(x).
   \end{equation}
   The last three equations \eqref{eq : 3 equations} together with equation \eqref{eq : contact vector field} imply that the contact vector field $Y$ associated to the $\alpha$-Hamiltonian function $\overline{h}$ satisfies for all $t\in\R$
   \[Y(\phi_\alpha^t(x))=\overline{h}(x)R_\alpha(\phi_\alpha^t(x)).\] 
   Therefore the map $ [0,1]\to \sphere{2n+1}$, $t\mapsto \phi_\alpha^{\overline{h}(x)t
}(x)$ solves the differential equation $\gamma'(t)=Y(\gamma(t))$ with initial condition $\gamma(0)=x$. By unicity of solutions of such equations $\overline{\phi_h^t}(x)=\phi_\alpha^{\overline{h}(x)t
}(x)$ for all $t\in[0,1]$ and the claim \eqref{eq : translated point equadiff} follows. In particular we have proved that for all $T\in(0,1]$ 
\begin{equation}\label{eq : spectre}
    \spec^\alpha(\overline{\phi_h^T})=\left\{T(\pi r^2-\varepsilon)+2\pi m\ |\ m\in\Z\right\}\bigcup\left\{2\pi m\ |\ m\in\Z\right\}.
    \end{equation}

%Putting $\phi_h:=(\phi_h^t)_{t\in[0,1]}$ we claim that 
%\begin{equation}
 %   c([(\phi_h)])=\pi r^2-\varepsilon
%\end{equation}

\textbf{Step 4 : Computation of the selector and conclusion} Let us denote by $\widetilde{\phi_h^T}:=[(\phi_h^{tT})]\in\widetilde{\conto}(L_k^{2n+1},\xi_k)$ for all $T\in[0,1]$. Proposition \ref{prop : spectral selector for L_k^{2n+1}} ensures that $c_k(\widetilde{\phi_h^0})=0$ and that the map $[0,1]\to\R,\ T\mapsto c_k(\widetilde{\phi_h^T})\in \overline{\spec^{\alpha_k}}(\widetilde{\phi_h^T})=\spec^\alpha(\overline{\phi_h^T})$ is continuous. Therefore there is $a\in\{T(\pi r^2-\varepsilon),0\}$ so that $c_k(\widetilde{\phi_h^T})=Ta$ for all $T\in[0,\delta]$ where $\delta:=\min\{1,\frac{2\pi}{\pi r^2-\varepsilon}\}$.
 \begin{center} \includegraphics[scale=0.5]{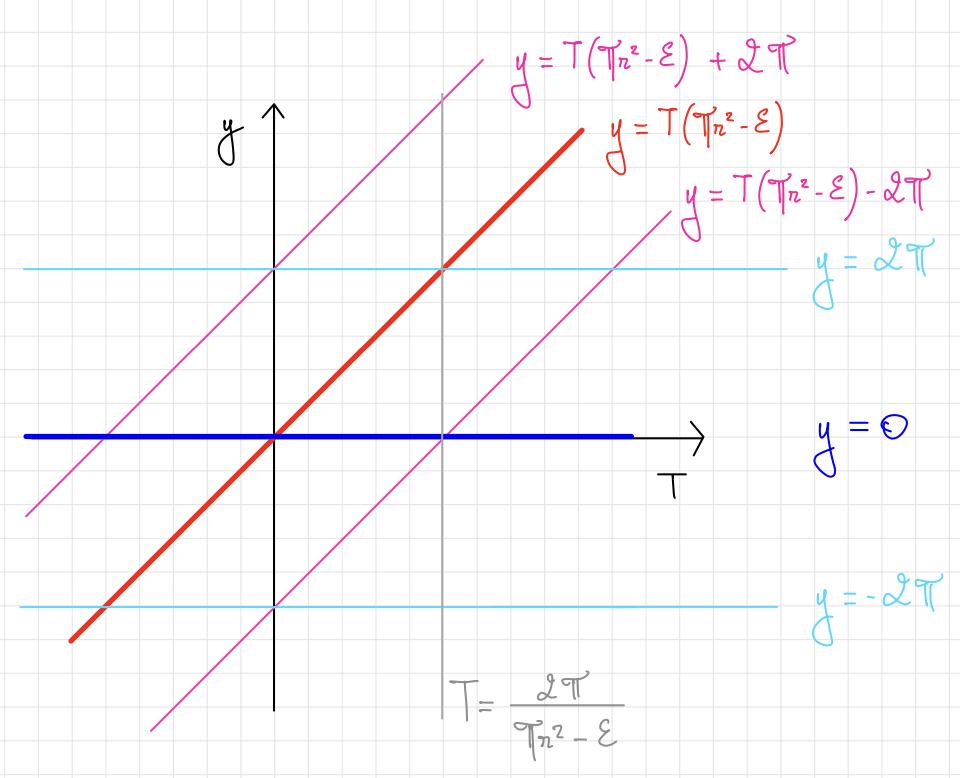}
 
 {\small{Figure 2: The graphs of $f_j : T\mapsto T(\pi r^2-\varepsilon)+2\pi j$ and $g_j : T\mapsto 2\pi j$ for $j\in\{-1,0,1\}$}}
\end{center}
Since $\id\cleq \widetilde{\phi_h^T}$ and $\widetilde{\phi_h^T}\ne\id$, for $T>0$, we use Proposition \ref{prop : spectral selector for L_k^{2n+1}} again to get that $c_k(\widetilde{\phi_H^T})>0$ and thus
\begin{equation}\label{eq : calcul du lower bound} 
c_k(\widetilde{\phi_h^T})=T(\pi r^2-\varepsilon)\text{ for all }T\in[0,\delta].
\end{equation}
 The relation \eqref{eq : calcul du lower bound} actually holds for all $T\in[0,1]$ since by
\cite[Thm 1]{Karshon_2005} the radius  $r\leq\sqrt{2}$ and so $\delta=1$ (see also Remark \ref{rem : introduction}). The $\alpha_k$-Hamiltonian function $h$ being supported in 
  $B_k(r):=\pi_k^{-1}(\Psi(\mathbb{B}^{2n},(r)))$ we get by definition that
   \[C_k(B_k(r))\geq c_k(\widetilde{\phi_h^1})=\pi r^2-\varepsilon.\]
   Since $\varepsilon>0$ can be choosen as small as we want this finishes the proof. \end{proof}
\subsection{Estimation of the upper bound}\label{sec : computation upper bound}
The aim of this subsection is to prove the upper bound in Theorem \ref{thm : contact capacity}, for all $n\in\N_{>0}$ and integer $k\geq 2$, which we recall in the following Proposition. 

\begin{prop}\label{prop : lower bound}
    Consider $\Psi : (\mathbb{B}^{2n}(\lambda R),\omega_0)\hookrightarrow (\CP^n,\omega)$ a symplectic embedding where $R>0$ and $\lambda>\sqrt{3}$. Then for all $r\in (0,R]$
    \[C_k(B_k(r))\leq\pi r^2\]
    where $B_k(r):=\pi_k^{-1}(\Psi(\mathbb{B}^{2n}(r)))$. 
\end{prop}

Let us first consider $(\R^{2}=\R^{2}\times\{0\},dx_1\wedge dy_1)$ the symplectic submanifold of $(\R^{2n},\sum\limits_{i=1}^n dx_i\wedge dy_i)$. We recall a Lemma from \cite[Lemma 4.2]{Hind_2013}.
\begin{lem}[\cite{Hind_2013}]\label{lem : displacement}
    For every $\delta>0$ there exists an Hamiltonian function $H : \R^{2}\to\R_{\geq 0}$ supported in $\mathbb{B}^2(\sqrt{2+2\delta})$ so that $\max H \leq (\pi +\delta)$ and $\psi_H^1(\mathbb{B}^{2}(1))\bigcap \mathbb{B}^{2}(1)=\emptyset$.  
\end{lem}

As a direct corollary we have.

\begin{cor}\label{cor : displacement}
    For every $r, \delta>0$ there exists an Hamitlonian function $H : \R^2\to\R_{\geq 0}$ supported in $\mathbb{B}^2(r\sqrt{2+2\delta})$ so that $\max H\leq r^2(\pi+\delta)$ and  $\psi_H^1(\mathbb{B}^{2}(r))\bigcap \mathbb{B}^2(r)=\emptyset$. 
\end{cor}

\begin{proof}
    Let $H$ be a Hamiltonian function that satisfies the requirements of Lemma \ref{lem : displacement} with respect to $\delta>0$ and consider the diffeomorphism $\psi_r :\R^2\to\R^2,\ x\mapsto r x$. A direct computation shows that the isotopy $(\psi_r\circ\psi_t\circ\psi_r^{-1})$ is contained in $\hamc(\R^2,\omega_0)$ and is generated by the Hamiltonian function $H_r:=r^2 H\circ(\psi_r)^{-1}$ whose support is equal to $(\psi_r)(\supp(H))$. Therefore $H_r$ satisfies the requirements of Corollary \ref{cor : displacement} with respect to $\delta>0$ and $r>0$. \end{proof}
    
   % for any isotopy $(\psi_t)\subset\hamc(\R^{2},\omega_0)$ generated by $H : [0,1]\times\mathbb{R}^2\to\R$

  %  Let $H : \R^2\to\R_{\geq 0}$ be a Hamiltonian function that satisfies the requirements of Lemma \ref{lem : displacement}. Then the time $1$ of the isotopy generated by the Hamiltonian function $H_R$ clearly satisfies the requirements of Corollary \ref{cor : displacement} since $\psi_R(\mathbb{B}^2(A))=\mathbb{B}^2(AR)$ for all $A>0$
    
    %path of Hamiltonian symplectomorphisms $(\psi_R\circ\psi_h^t\circ\psi_{R}^{-1})$ generated by the Hamiltonian function $H_R:=R^2 H\circ(\psi_R)^{-1}$ 

We are now ready to prove Proposition \ref{prop : lower bound}.

\begin{proof}[Proof of Proposition \ref{prop : lower bound}]

\textbf{Step 1 : Construction in $(\mathbb{B}^{2n}(\lambda R),\omega_0)$}

    Let $r\in(0,R]$, $\delta>0$ be small enough so that $\sqrt{3+2\delta}<\lambda$  and $H :\R^2\to\R_{\geq 0}$ be an Hamiltonian function supported in $\mathbb{B}^{2}(r\sqrt{2+2\delta})$ so that $\max H\leq r^2(\pi+\delta)$ and $\psi_H^1(\mathbb{B}^2(r))\bigcap \mathbb{B}^2(r)=\emptyset$ whose existence is guaranteed by Corollary \ref{cor : displacement}. The function $K :\R^{2n}\to \R_{\geq 0}, ((x_1,y_1),\cdots, (x_n,y_n))\mapsto H(x_1,y_1)$, even if non compactly supported, generates a vector field $X_k$, i.e. $\omega_0(X_k,\cdot)=-dK$, which is complete, and whose flow $(\psi_K^t)\subset\ham(\R^{2n},\omega_0)$ satisfies $\psi_K^t=\psi_H^t\times \id_{\R^{2(n-1)}}$ for all $t\in[0,1]$. Therefore \[\psi_K^1\left(\mathbb{B}^{2n}(r)\right)\bigcap\mathbb{B}^{2n}(r)\subset \left(\psi_H^1\left(\mathbb{B}^2(r)\right)\times\mathbb{B}^{2(n-1)}(r)\right)\bigcap \mathbb{B}^2(r)\times\mathbb{B}^{2(n-1)}(r)=\emptyset.\]
Moreover  $\underset{t\in[0,1]}\bigcup\psi_K^t(\mathbb{B}^{2n}(r))$ is contained in $\mathbb{B}^{2n}(r\sqrt{3+2\delta})$ since \[\psi_K^t(\mathbb{B}^{2n}(r))\subset\psi_H^t(\mathbb{B}^{2}(r))\times \mathbb{B}^{2(n-1)}(r)\subset \mathbb{B}^{2}(r\sqrt{2+2\delta})\times \mathbb{B}^{2(n-1)}(r)\]
and the latter set is clearly contained inside $\mathbb{B}^{2n}(r\sqrt{3+2\delta})$. Given that $\sqrt{3+2\delta}<\lambda$ one can construct $\rho :\R^{2n}\to [0,1]$ a smooth cutoff function supported in $\mathbb{B}^{2n}(\lambda R)$ which is constantly equal to $1$ on $\mathbb{B}^{2n}(r\sqrt{3+2\delta})$ so that the Hamiltonian function $\widetilde{K}:=\rho K$ generates a path of Hamiltonian symplectomorphisms $(\psi_t)$ satisfying
\begin{enumerate}
    \item $\widetilde{K}$ and $(\psi_t)$ are both supported in $\mathbb{B}^{2n}(\lambda R)$
    \item $\widetilde{K}\geq 0$ and $\max \widetilde{K}\leq r^2(\pi+\delta)$
    \item $\psi_1\left(\mathbb{B}^{2n}(r)\right)\cap\mathbb{B}^{2n}(r)=\emptyset$.\\
\end{enumerate}
    
    \textbf{Step 2 : Import the construction to $(\CP^n,\omega)$}
 %  Seeing $\widetilde{K}$ as a compactly supported Hamiltonian function of the open symplectic manifold $(\mathbb{B}^{2n}(\lambda R),\omega_0)$ and its flow $(\psi_t)$ as an isotopy in $\hamc(\mathbb{B}^{2n}(\lambda R),\omega_0)$ one can 
 %   transport this construction to $(\CP^n,\omega)$, thanks to the first point of Lemma \ref{lem : contact lift}. More precisely
 A direct computation shows that the Hamiltonian function
   \[G : \CP^n\to\R_{\geq 0},\quad \quad  
    x\mapsto \left\{
    \begin{array}{ll}
       \widetilde{K}(\Psi^{-1}(x)) & \text{ if } x\in \Psi(\mathbb{B}^{2n}(\lambda R))\\
        0 & \text{ otherwise}
    \end{array}
\right.\]
on $(\CP^n,\omega)$, generates an isotopy $(\psi_G^t)\subset\ham(\CP^n,\omega)$ that satisfies
\[[0,1]\times\CP^n\to\CP^n,\quad \quad (t,x)\mapsto \left\{
    \begin{array}{ll}
       \Psi(\psi_t(\Psi^{-1}(x))) & \text{ if } x\in \Psi(\mathbb{B}^{2n}(\lambda R))\\
        x & \text{ otherwise.}
    \end{array}
\right. \] We therefore clearly have: 
\begin{enumerate}
\item $G\geq 0$ and $\max G\leq r^2(\pi+\delta)$
%\item they are both supported in $B(\lambda R)$
   \item $\psi_G^1(B(r))\cap B(r)=\emptyset$, where $B(r):=\Psi(\mathbb{B}^{2n}(r))$.\\
\end{enumerate}

\textbf{Step 3 : Lift the construction to $(L_k^{2n+1},\xi_k)$} We lift this construction to $(L_k^{2n+1},\ker\alpha_k)$, thanks to Lemma \ref{lem : contact lift}. More precisely the $\alpha_k$-Hamiltonian function 
$h : L_k^{2n+1}\to\mathbb{R}_{\geq 0}$, $x\mapsto G(\pi_k(x))$ generates a flow of $\sphere{1}$-equivariant contactomorphisms $(\phi_h^t)\subset\conto(L_k^{2n+1},\xi_k)$ satisfying :
\begin{enumerate}
    \item $h\geq 0$ and $\max h\leq r^2(\pi+\delta)$
    \item $\phi_h^1(B_k(r))\cap B_k(r)=\emptyset$, where $B_k(r):=\pi_k^{-1}(\Psi(\mathbb{B}^{2n}(r)))$.
\end{enumerate}
Indeed, if we suppose by contradiction that $\phi_h^1(B_k(r))\cap B_k(r)\ne\emptyset$ then we get the following absurdity \[\emptyset\ne\pi_k\left(\phi_h^1(B_k(r))\cap B_k(r)\right)\subset \pi_k\left(\phi_h^1(B_k(r))\right)\cap \pi_k\left(B_k(r)\right)=\psi_G^1(B(r))\cap B(r)=\emptyset.\]

\textbf{Step 4 : Conclusion } Therefore thanks to Theorem \ref{thm : capacite-energie}
\[c_k(B_k(r))\leq c_k(\widetilde{\phi}^{-1})+c_k(\widetilde{\phi})\]
where $\widetilde{\phi}:=[(\phi_h^t)]\in\widetilde{\conto}(L_k^{2n+1},\xi_k)$. Finally, since the $\alpha_k$-Hamiltonian function $L_k^{2n+1}\to\R_{\leq 0},\ x\mapsto -h(x)$, generates the contact isotopy $(\phi_h^t)^{-1}$ which represents $\widetilde{\phi}^{-1}$, the last point of Proposition \ref{prop : spectral selector for L_k^{2n+1}} implies that 
\[ c_k(\widetilde{\phi})+c_k(\widetilde{\phi}^{-1}) \leq \int_0^1\max h\ud t-\int_0^1 \min h\ud t\leq r^2(\pi+\delta).\]
Letting $\delta$ goes to $0$ finishes the proof.
\end{proof}

\section{Order spectral selectors for strongly orderable prequantizations}\label{sec : spectral selectors for strongly orderable prequantization}

\subsection{Order spectral selectors on Legendrian isotopy class}\label{sec : strongly orderability}
Let $(M,\xi)$ be a cooriented contact manifold, not necessarily closed. As in the closed case in Section \ref{sec : spectral selectors}, a contact isotopy, that we denote by $(\phi_t)$, is a $[0,1]$-family of contactomorphisms $\{\phi_t\}\subset\cont(M,\xi)$ so that  $[0,1]\times M\to M$ is smooth, and we denote by $\conto(M,\xi)$ the set of contactomorphisms isotopic to the identity, and by $\contoc(M,\xi)$ the ones that are isotopic to the identity through a compactly supported isotopy. Similarly, a $[0,1]$-family of closed Legendrians $(\Lambda_t)\subset (M,\xi)$ is a Legendrian isotopy if and only if there exists a compactly supported contact isotopy $(\phi_t)\subset\contoc(M,\xi)$ parametrizing it, i.e.  $\phi_t(\Lambda_0)=\Lambda_t$ and $\phi_t$ is compactly supported. Note that all the Legendrians we deal with in this paper are closed (even if the contact manifold itself may not). Therefore the isotopy class of a closed Legendrian $\Lambda\subset M$, that we denote by $\Leg(\Lambda)$, is an homogeneous space $\contoc(M,\xi)/\mathrm{Stab}(\Lambda)$ and inherits the $C^1$-topology of $\contoc(M,\xi)$. We denote by $\uLeg(\Lambda)$ its universal cover and $\Pi : \uLeg(\Lambda)\to\Leg(\Lambda)$ the natural projection. Note that we have a natural left action of $\conto(M,\xi)$ (resp. $\widetilde{\conto}(M,\xi)$\footnote{the set of equivalence classes of contact isotopy starting at the identity, where two such contact isotopies are identified if they are homotopic relatively to endpoints}) on $\Leg(\Lambda)$ (resp. $\uLeg(\Lambda)$). 

If $\alpha$ is a contact form whose Reeb vector field is complete, then as for contactomorphisms we have a notion of spectrum. More precisely let $\tLambda_1,\tLambda_0\in\uLeg(\Lambda)$ and denote by $\Lambda_1$ and $\Lambda_0$ their respective projections to $\Leg(\Lambda)$ then 
\[\spec^\alpha(\tLambda_1,\tLambda_0)=\{t\in\R \ |\ \phi_\alpha^t(\Lambda_0)\cap \Lambda_1\ne\emptyset\}.\]

 A Legendrian isotopy $(\Lambda_t)$ is said to be non-negative (resp. positive) if for some, and thus all, contact form $\alpha$ supporting $\xi$ we have $\alpha(\frac{d}{dt}\phi_t(x))\geq 0$ (resp. $>0$) for all $x\in\Lambda_0$ and any compactly supported contact isotopy $(\phi_t)$ parametrizing it. We say, following \cite{EP00}, that $\uLeg(\Lambda)$ is orderable if there does not exist any positive contractible loop of Legendrians in $\Leg(\Lambda)$. Assuming that $\uLeg(\Lambda)$ is orderable implies that the binary relation $\cleq$, defined by $\tLambda_0\cleq\tLambda_1$ if and only if there exists a path $(\tLambda_t)$ joining $\tLambda_0$ to $\tLambda_1$ such that $(\Pi(\tLambda_t))$ is a non-negative Legendrian isotopy, is a partial order.

For any contact form $\alpha$ supporting $\xi$ whose Reeb vector field is complete  in \cite{allais2023spectral} the authors defined a map $\ell_+^\alpha : \uLeg(\Lambda)\times\uLeg(\Lambda) \to \R\cup\{-\infty\}$
\[
(\tLambda_1,\tLambda_0) \mapsto \inf\left\{t\in\R\ | \tLambda_1\cleq\widetilde{\phi_{\alpha}^t}\cdot\tLambda_0\right\}=\inf\left\{\underset{[0,1]\times \Lambda_0}\max \alpha\left(\frac{d}{dt}\phi_t(x)\right)\ |\ [(\phi_t)]\cdot\tLambda_0=\tLambda_1\right\},
\] where $\Lambda_0:=\Pi(\tLambda_0)$, and showed the following: 
\begin{prop}[\cite{allais2023spectral}]\label{prop : spec selector legendrian}
The map $\ell_+^\alpha$ takes value in $\R$ if and only if $\uLeg(\Lambda)$ is orderable. Moreover if $\uLeg(\Lambda)$ is orderable then for any $\tLambda_0,\tLambda_1,\tLambda_2\in\uLeg(\Lambda)$ and any $\tphi\in\widetilde{\conto}(M,\xi)$
\begin{enumerate}
    \item $\ell_+^\alpha(\tLambda_0,\tLambda_0)=0$
    \item $\tLambda_0\cleq\tLambda_1$ implies that $\ell_+^\alpha(\tLambda_1,\tLambda_0)\geq 0$ and $\ell_+^\alpha(\Lambda_0,\Lambda_1)\leq 0$
    \item $\tLambda_0\cleq\tLambda_1$ and $\tLambda_0\ne\tLambda_1$ implies that $\ell_+^\alpha(\tLambda_1,\tLambda_0)>0$
    \item $\ell_+^\alpha(\tLambda_2,\tLambda_0)\leq \ell_+^\alpha(\tLambda_2,\tLambda_1)+ \ell_+^\alpha(\tLambda_1,\tLambda_0)$
    \item $\ell_+^\alpha(\tphi\cdot\tLambda_1,\tphi\cdot\tLambda_0)=\ell_+^{\varphi^*\alpha}(\tLambda_1,\tLambda_0)$ where $\varphi=\Pi(\tphi)$
    \item For any contact isotopy $(\phi_t)$ representing $\tphi$
    \[\int_0^1\underset{x\in\Pi(\tLambda_0)}\min \alpha\left(\frac{d}{dt}\phi_t(x)\right)\ud t\leq \ell_+^\alpha(\tphi\tLambda_0,\tLambda_0)\leq \int_0^1\underset{x\in\Pi(\tLambda_0)}\max \alpha\left(\frac{d}{dt}\phi_t(x)\right)\ud t\]
    \item $\ell_+^\alpha(\tLambda_1,\tLambda_0)\in\spec^\alpha(\tLambda_1,\tLambda_0)$.
\end{enumerate}
\end{prop}

\subsection{Spectral selectors for strongly orderable closed contact manifold}\label{sec : spectral selectors for strongly orderable prequantization 2}
Let $(M,\xi=\ker\alpha)$ be a closed cooriented contact manifold and consider its contact product $(M\times M\times\R,\Xi=\ker(\beta:=\mathrm{pr}_2^*\alpha-e^\theta\mathrm{pr}_1^*\alpha))$ as in Section \ref{sec : intro ns prequantization}. To any $[(\varphi_t)]=\tphi\in\widetilde{\conto}(M,\xi)$ one can associate $[(\Phi_t)]=\tPhi\in\widetilde{\conto}(M\times M\times\R,\Xi)$ defined by 
\begin{equation}\label{eq : graph of contact isotopy}\Phi_t : (x_1,x_2,\theta)\mapsto (x_1,\varphi_t(x_2),\theta+g_t(x_2)),
\end{equation}
where $g_t$ is the conformal factor of $\varphi_t$ with respect to $\alpha$, i.e. $\phi_t^*\alpha=e^{g_t}\alpha$. We say that $\tPhi$ is the graph of $\tphi$. Direct computations show that 
\begin{equation}\label{eq : comparaison spectre}
\spec^\alpha(\tphi)=\spec^\beta(\tPhi\cdot \Delta,\Delta),
\end{equation}
where we have identified the Legendrian $\Delta=\{(x,x,0)\ |\ x\in M\}\subset (M\times M\times\R,\ker \beta)$ with the equivalence class of the constant isotopy at $\Delta$ in $\uLeg(\Delta)$, and that for all $t\in [0,1]$ and $(x_1,x_2,\theta)\in M\times M\times \R$
\begin{equation}\label{eq : ham graph}
\beta\left(\frac{d}{dt}\Phi_t(x_1,x_2,\theta)\right)=\alpha\left(\frac{d}{dt}\phi_t(x_2)\right).
\end{equation}

Suppose that in addition to be closed $(M,\xi=\ker\alpha)$ is also strongly orderable, i.e. $\uLeg(\Delta)$ is orderable, then using Proposition \ref{prop : spec selector legendrian} we can define the map
\[C_+^\alpha : \widetilde{\conto}(M,\xi)\to \R\quad \quad \tphi\mapsto \ell_+^\beta(\tPhi\cdot \Delta,\Delta), \text{ where }\tPhi\text{ is the graph of }\tphi\]   which satisfies moreover the following:

\begin{prop}\label{prop : spec selector strongly orderable}
    Suppose that $(M,\xi=\ker\alpha)$ is a closed strongly orderable contact manifold. Then for any $\tphi,\tpsi\in\widetilde{\conto}(M,\xi)$ denoting by $\tPhi,\widetilde{\Psi}\in\widetilde{\conto}(M\times M\times\R,\ker\beta)$ their corresponding graphs 
    \begin{enumerate}
        \item $C_+^\alpha(\id)=0$
        \item $\tpsi\cleq\tphi$ implies that $C_+^\alpha(\tpsi)\leq C_+^\alpha(\tphi)$
        \item $\id\cleq \tphi$ and $\tphi\ne\id$ implies that $C_+^\alpha(\tphi)>0$
         \item $C_+^\alpha(\tpsi\tphi)\leq \ell_+^{\Psi^*\beta}(\tPhi\Delta,\Delta)+C_+^\alpha(\tpsi)$ where $\Psi:=\Pi(\tPsi)\in\conto(M\times M\times\R,\Xi)$
        \item For any contact isotopy $(\phi_t)$ representing $\tphi$ 
        \[\int_0^1\min \alpha\left(\frac{d}{dt}\phi_t(x)\right)\ud t\leq C_+^\alpha(\tphi)\leq \int_0^1\max \alpha\left(\frac{d}{dt}\phi_t(x)\right)\ud t\]
       
        \item $C_+^\alpha(\tphi)\in\spec^\alpha(\tphi)$.
    \end{enumerate}
\end{prop}
\begin{proof}\
   \begin{enumerate}
   \item $C_+^\alpha(\id)=\ell_+^\beta(\Delta,\Delta)=0$
       \item By the triangle inequality of Proposition \ref{prop : spec selector legendrian} we have    \[\begin{aligned}\ell_+^\beta(\tPhi\cdot \Delta,\Delta)-\ell_+^\beta(\tPsi\cdot\Delta,\Delta)&\geq\ell_+^\beta(\tPhi(\Delta),\Delta)-\ell_+^\beta(\tPsi\cdot\Delta,\tPhi\cdot\Delta)-\ell_+^\beta(\tPhi\cdot\Delta,\Delta)\\
       &=-\ell_+^\beta(\tPsi\cdot\Delta,\tPhi\cdot\Delta).
       \end{aligned}\]
       If we suppose that $\tpsi\cleq\tphi$ this implies, using \eqref{eq : ham graph}, that $\tPsi\cdot\Lambda\cleq\tPhi\cdot\Lambda$ and thus using the monotonicity property of $\ell_+^\beta$ we get the desired result
       \[C_+^\alpha(\tphi)-C_+^\alpha(\tpsi)\geq -\ell_+^\beta(\tPsi\cdot\Delta,\tPhi\cdot\Delta)\geq 0.\]
       \item Let $\tphi$ be such that $\id\cleq\tphi$ and $\tphi\ne\id$. Let $(\phi_t)$ be a non-negative contact isotopy so that $[(\phi_t)]=\tphi$ and consider its graph $(\Phi_t)$ which is also non-negative contact isotopy by \eqref{eq : ham graph}. Since $\tphi\ne\id$ there exists $T\in[0,1]$ and $x\in M$ so that $\phi_T(x)\ne x$ which implies that $\Phi_T(\Delta)\ne \Delta$. Therefore using the non-degeneracy and monotonicity of $\ell_+^\beta$ we get that \[C_+^\alpha(\tphi)=\ell_+^\beta(\tPhi_1\cdot\Delta,\Delta)\geq \ell_+^\beta(\tPhi_T\cdot \Delta,\Delta)>0\]
       
       \item By the triangle inequality property and the naturality of $\ell_+^\beta$ we get that
       \[\ell_+^\beta(\tPsi\tPhi\cdot\Delta,\Delta)\leq \ell_+^\beta(\tPsi\tPhi\cdot\Delta,\tPsi\cdot\Delta)+\ell_+^\beta(\tPsi\cdot\Delta,\Delta)=\ell_+^{\Psi^*\beta}(\tPhi\cdot\Delta,\Delta)+C_+^\alpha(\tpsi).\]
        \item The estimates come from the one of Proposition \ref{prop : spec selector legendrian} combined with \eqref{eq : ham graph}.
        \item The spectrality of $C_+^\alpha$ follows from \eqref{eq : comparaison spectre} and the spectrality of $\ell_+^\beta$.
   \end{enumerate} \end{proof}
As a corollary we get the following.
\begin{cor}\label{lem : inegalite triangulaire strongly orderable prequantization}
    Suppose that $\tpsi\in\widetilde{\conto}(M,\ker\alpha)$ can be represented by a path $(\psi_t)$ so that $\psi_t^*\alpha=\alpha$ for all $t\in[0,1]$ then 
    \[C_+^\alpha(\tpsi\tphi)\leq C_+^\alpha(\tpsi)+C_+^\alpha(\tphi).\]
\end{cor}
\begin{proof}
    Fist note that for any $\psi\in\conto(M,\ker\alpha)$ if $g:M\to\R$ is the conformal factor of $\psi$ then the conformal factor of its graph $\Psi\in\conto(M\times M\times\R,\ker\beta)$ is given by $G: M \times M\times\R\to\R,\ (x_1,x_2,\theta)\mapsto g(x_2)$. Therefore if $\psi^*\alpha=\alpha$ it implies that $\Psi^*\beta=\beta$. The proof of the corollary follows now directly from the triangle inequality of Proposition \ref{prop : spec selector strongly orderable}.
\end{proof}
\begin{rem}
    If $\psi\in\conto(M,\xi)$ is such that $\psi^*\alpha=\alpha$ then $\psi\circ\phi_\alpha^t=\phi_\alpha^t\circ\psi$ for all $t\in\R$. Indeed, it implies that $\psi^*R_\alpha=R_\alpha$
\end{rem}

\subsection{Order spectral selectors for strongly orderable prequantizations}\label{sec : avant derniere section} Let $(M,\ker\alpha)$ be a cooriented contact manifold and $\Gamma$ a finite group acting freely on $M$ by $\alpha$-strict contactomorphism, i.e. $\gamma^*\alpha=\alpha$ for any $\gamma\in G$. Then $\alpha$ descends to a contact form $\alpha'$ on $M':=M/\Gamma$ and the subgroup  $\conto^\Gamma(M,\ker\alpha)\subset\conto(M,\ker\alpha)$ of contactomorphisms isotopic to the identity that are equivarient with respect to the action of $\Gamma$ is a covering space over $\conto(M',\ker\alpha')$. In particular any contact isotopy starting at the identity in $\conto(M',\ker\alpha')$ admits a unique lift to $\conto^\Gamma(M,\ker\alpha)$ starting at the identity and so we deduce a lifting map \begin{equation}\label{eq : lifting map}\widetilde{\conto}(M',\ker\alpha')\to\widetilde{\conto}(M,\ker\alpha).
\end{equation}

Recall that a contact manifold $(M_1,\xi_1')$ is a prequantization over a symplectic manifold $(W,\omega')$ if there exists a $\mathbb{T}^1$-principal bundle $\pi_1' : M_1\to W$, where $\mathbb{T}^1=\R/2\pi\Z$, and a contact form $\alpha_1'$ supporting $\xi_1'$ so that the Reeb flow of $\alpha_1'$ is Zoll of minimal period $2\pi$, induces the $\mathbb{T}^1$-action and satisfies $\pi_1'^*\omega'=\ud \alpha'_1$.

From now on until the end of the paper, we drop all the primes from our notations, and  $(M_1,\xi_1=\ker\alpha_1)$ is any closed prequantization over some symplectic manifold $(W,\omega)$ and $\alpha_1$ is the contact form whose Reeb flow is Zoll of minimal period $2\pi$, induces the $\mathbb{T}^1=\R/2\pi\Z$ action on $M_1$ and satisfies $\pi_1^*\omega=\alpha_1$, where $\pi_1 : M_1\to W=M_1/\mathbb{T}^1$ is the natural projection. For any $k\in\N_{>0}$, $\Z_k$ is a finite group acting on $(M_1,\xi_1)$, by $[j]\cdot x=\phi_{\alpha_1}^{\frac{2\pi j}{k}}(x)$ for all $j\in\Z_k,\ x\in M_1$, and we denote by $M_k:=M_1/\Z_k$ the quotient manifold, by $\xi_k$ the contact distribution given by the kernel of the induced contact $\alpha_k$ and by $\pi_k : M_k\to W$ the natural projection (see \ref{sec : intro ns prequantization}). We have the following Lemma.
\begin{lem}\label{lem : a valeurs dans le lift du spectre}
If $(M_1,\xi_1)$ is a strongly orderable prequantization then for any $k\in\N_{>0}$ \begin{enumerate}
    \item $(M_k,\xi_k)$ is strongly orderable  \item for any $\tphi\in\widetilde{\conto}(M_k,\xi_k)$ denoting by $\tpsi\in\widetilde{\conto}(M_1,\xi_1)$ its lift we have
\[C_+^{\alpha_k}(\tphi)\geq C_+^{\alpha_1}(\tpsi).\]
\item $\left\lceil C_+^{\alpha_k}(\tpsi\tphi\tpsi^{-1}) \right\rceil_{\frac{2\pi}{k}}=\left\lceil C_+^{\alpha_k}(\tphi)\right\rceil_{\frac{2\pi}{k}}$.
\end{enumerate}
\end{lem}

\begin{proof}
 In the following $(M_k\times M_k\times\R,\Xi_k=\ker\beta_k)$ denote the contact product of $(M_k,\xi_k=\ker\alpha_k)$ with itself and $\beta_k$ the product contact form associated to $\alpha_k$ for any $k\in\N_{>0}$ as in Section \ref{sec : intro ns prequantization}.
  \begin{enumerate} 
  \item Suppose that $(M_k,\xi_k)$ is not strongly orderable. It means that there exists a positive contractible loop of Legendrians $(\Lambda_t)\subset (M_k\times M_k\times\R,\Xi_k)$ based at $\Delta_k:=\{(x,x,0)\ |\ x\in M_k\}$. Consider $(\Lambda_t^s)$ a homotopy of this loop with fixed endpoint to the constant loop based at $\Delta_k$, i.e. $\Lambda_t^0=\Lambda_t$ and $\Lambda_t^1=\Lambda_1^s=\Lambda_0^s=\Delta_k$ for all $t,\ s\in[0,1]$, and consider $(\Phi_t^s)\subset \conto(M_k\times M_k\times\R,\Xi_k)$ so that $\Phi_t^s(\Delta_k)=\Lambda_t^s$ and $\Phi_0^s=\id$ for all $t,s\in[0,1]$ (for instance see \cite[Lemma 2.1]{allais2023spectral} for the construction of $(\Phi_t^s)$). Seeing $M_k\times M_k\times \R$ as the quotient of $M_1\times M_1\times \R$ by the action by strict contactomorphisms of $\Z_k\times \Z_k$, i.e. $([j_1],[j_2])\cdot (x_1,x_2,\theta)=([j_1]\cdot x_1,[j_2]\cdot x_2,\theta)$, the group $\conto^{\Z_k\times \Z_k}(M_1\times M_1\times\R,\Xi_1)$ is a covering space over $\conto(M_k\times M_k\times\R,\Xi_k)$ and therefore there exists a unique $(\Psi_t^s)\subset\conto^{\Z_k\times \Z_k}(M_1\times M_1\times\Xi_1)$ that lifts $(\Phi_t^s)$, i.e. $\Pi_k\circ\Psi_t^s=\Phi_t^s\circ\Pi_k$ where $\Pi_k : M_1\times M_1\times\R\to M_k\times M_k\times \R$ denotes the natural projection, and so that $\Psi_0^s=\id$ for all $s,t\in[0,1]$. It is clear that $(\Psi_t^1(\Delta_1))$ is a contractible loop in $\Leg(\Delta_1)$, and moreover since $\Pi_k^*\beta_k=\beta_1$ a direct computation shows that it is a positive loop which implies that $(M_1,\xi_1)$ is not strongly orderable.
  \item Let $\tphi\in\conto(M_k,\xi_k)$ and its graph $\tPhi\in\conto(M_k\times M_k\times\R,\Xi_k)$. Denoting by $\tpsi$ the lift $\tphi$ to $\conto(M_1,\xi_1)$,  it is easy to see that its graph $\tPsi\in\conto(M_1\times M_1,\times\R,\Xi_1)$ corresponds to the $\Z_k\times\Z_k$-equivariant lift of $\tPhi$. Consider any contact isotopy $(\eta_t)\subset\conto(M_k\times M_k\times\R,\Xi_k)$ starting at the identity such that $[(\eta_t)]\cdot\Delta_k=\tPhi\cdot\Delta_k$, and $(\overline{\eta}_t)\subset\conto^{\Z_k\times\Z_k}(M_1\times M_1\times\R,\Xi_1)$ its corresponding $\Z_k\times\Z_k$-equivariant lift starting at the identity. We claim that $[(\overline{\eta_t})]\cdot\Delta_1=\tPsi\cdot\Delta_1$. Admitting this claim for the moment finishes the proof since $\underset{p\in\Delta_1}\max\beta_1(\frac{d}{dt}\overline{\eta}_t(p))=\underset{p\in\Delta_k}\max\beta_k(\frac{d}{dt}\eta_t(p))$. To prove the claim take any two contact isotopies $(\Phi_t)$ and $(\Psi_t)$ representing $\tPhi$ and $\tPsi$ respectively, i.e. $[(\Phi_t)]=\tPhi$ and $[(\Psi_t)]=\tPsi$. Then the Legendrian isotopy $(\Pi_k(\Psi_t^{-1}\circ\overline{\eta}_t)\Lambda_1)=(\Phi_t^{-1}\circ\eta_t(\Lambda_k))$ is a contractible loop based at $\Delta_k$ in $\Leg(\Delta_k)$ since $[(\eta_t)]\cdot\Delta_k=\tPhi\cdot\Delta_k$. Therefore the Legendrian isotopy  $((\Psi_t^{-1}\circ\overline{\eta}_t)\Lambda_1)$ has to be itself a contractible loop in $\Leg(\Delta_1)$ which concludes the proof of our claim.
\item It is a direct consequence of \ref{lem : prequantization}.
  \end{enumerate}
\end{proof}

\section{Computation of the contact capacity in strongly orderable prequantizations}\label{sec : Contact capacity for strongly orderable prequantizations}

Taking the same notations as in Section \ref{sec : avant derniere section} above, suppose from now on that  $(M_1,\xi_1=\ker\alpha_1)$ is a closed strongly orderable prequantization over $(W,\omega)$. Thanks to Lemma \ref{lem : a valeurs dans le lift du spectre} the associated contact manifold $(M_k,\xi_k)$ is also strongly orderable, for all $k\in\N_{>0}$, and therefore one can define the spectral selectors $C_+^{\alpha_k} : \widetilde{\conto}(M_k,\xi_k)\to\R$ of Section \ref{sec : spectral selectors for strongly orderable prequantization 2}). Then, as in Section \ref{sec : Contact capacity}, for any $k\in\N_{>0}$\footnote{here $k$ is not necessarily greater than $1$} and any non-empty open set $U\subset M_k$ we denote by $\widetilde{\contoc}(U,\xi_k)$ the subgroup of $\widetilde{\conto}(M_k,\xi_k)$ whose elements can be represented by contact isotopies $(\phi_t)$ supported in $U$. We define similarly $C_{\alpha_k} :\{\text{non-empty open sets of } M_k\}\to\R_{\geq 0}\cup\{+\infty\}$ to be the map: 
\[U\mapsto C_{\alpha_k}(U):=\sup\{ C^{\alpha_k}_+(\tphi)\ |\ \tphi\in\widetilde{\contoc}(U,\xi_k)\}.\]
Recall that $\pi_k : M_k\to W$ denotes the natural projection for all $k\in\N_{>0}$. As for lens spaces we have the following:

\begin{thm}\label{thm : contact capacity for strongly orderable prequantizations}\
For all non-empty open sets $U, V\subset M_k$ and $\phi\in\conto(M_k,\xi_k)$
   \begin{enumerate}
   \item(monotonicity) $C_{\alpha_k}(U)\leq C_{\alpha_k}(V)$ if $U\subset V$
       \item(invariance by conjugation) $\lceil C_{\alpha_k}(U)\rceil_{\frac{2\pi}{k}}=\lceil C_{\alpha_k}(\phi(U))\rceil_{\frac{2\pi}{k}}$ 
       \end{enumerate}
     \begin{enumerate}
       \item[(3)](lower bound) if $\Psi : (\mathbb{B}^{2n}(R),\omega_0)\hookrightarrow (W,\omega)$ is a symplectic embedding for some $R\in(0,\sqrt{2}]$ then  $C_{\alpha_k}(B_k(r))\geq \pi r^2$ for all $r\in(0, R]$ where $B_k(r):=\pi_k^{-1}(\Psi(\mathbb{B}^{2n}(r)))$,
       \item[(4)](upper bound) if $\Psi : (\mathbb{B}^{2n}(\lambda R),\omega_0)\hookrightarrow (W,\omega)$ is a symplectic embedding for some $R>0$ and $\lambda>\sqrt{3}$ then  $C_{\alpha_k}(B_k(r))\leq\pi r^2$ for all $r\in(0,R]$ where $B_k(r):=\pi_k^{-1}(\Psi(\mathbb{B}^{2n}(r)))$.
    \end{enumerate}
\end{thm}
To prove the monotonicity and the invariance by conjugation properties one can argue exactly as in Section \ref{sec : Contact capacity}. Also, note that as in Section \ref{sec : Contact capacity}, one deduces immediatly Theorem \ref{thm : ns for prequantization} from Theorem \ref{thm : contact capacity for strongly orderable prequantizations}.  It thus remains to prove the last two estimates and the proof will mimick the proof we have given in Section \ref{sec : computation of the capacity} with minor modifications. 

\subsection{Computation of the lower bound} Let us first prove the lower bound that we restate below.

\begin{prop}\label{prop : lower bound for strongly orderable prequantization}
   If $\Psi : (\mathbb{B}^{2n}(R),\omega_0)\hookrightarrow (W,\omega)$ is a symplectic embedding for $R\in(0,\sqrt{2}]$ then $C_{\alpha_k}(B_k(r))\geq\pi r^2$ for all $r\in (0,R]$
    where $B_k(r):=\pi_k^{-1}(\Psi(\mathbb{B}^{2n}(r)))$.  
\end{prop}

Before proving Proposition \ref{prop : lower bound for strongly orderable prequantization} let us state the following Lemma.
\begin{lem}\label{lem : capacite croissante}
Let $U$ be an open subset of $(W,\omega)$ then for any $k\in\N_{>0}$
\[C_{\alpha_k}(\pi_k^{-1}(U))\geq C_{\alpha_1}(\pi_1^{-1}(U)).\]
\end{lem}

\begin{proof}
    For all $\tphi_1\in\widetilde{\contoc}(\pi_1^{-1}(U))$ one can easily construct $\tpsi_k\in\widetilde{\contoc}(\pi_k^{-1}(U))$ so that the lift of $\tpsi_k$ to $\widetilde{\contoc}(M_1,\xi_1)$ that we denote by $\tpsi_1$ satisfies $\tphi_1\cleq \tpsi_1$ and $\tpsi_1\in\widetilde{\contoc}(\pi_1^{-1}(U))$. Moreover thanks to Lemma \ref{lem : a valeurs dans le lift du spectre} and the monotonicity of $C_+^{\alpha_1}$ we have 
    \[C_+^{\alpha_k}(\tpsi_k)\geq C_+^{\alpha_1}(\tpsi_1)\geq C_+^{\alpha_1}(\tphi_1)\]
    which concludes the proof.
\end{proof}

\begin{proof}[Proof of Proposition \ref{prop : lower bound for strongly orderable prequantization}] The first 4 steps of the proof will follow the same steps as the proof of Proposition \ref{prop : lower bounds}.\\
   \textbf{Step 1 : } For any $r\in(0,R]$ and $0<\varepsilon<\pi r^2$, consider the same function $ f:\R_{\geq 0}\to\R_{\geq 0}$ as in the first step of the proof of Proposition \ref{prop : lower bounds} and consider the same map $H : \mathbb{B}^{2n}(r)\to\R,\ x\mapsto f(\|x\|^2)$.\\
   \textbf{Step 2 : } Use the symplectic embedding $\Psi$ to import the construction to $(W,\omega)$, i.e. \[G:W\to\R \quad  \quad 
    x\mapsto \left\{
    \begin{array}{ll}
       H(\Psi^{-1}(x)) & \text{ if } x\in \Psi(\mathbb{B}^{2n}(r))\\
        0 & \text{ otherwise}
    \end{array}.
\right. \]
   \textbf{Step 3 : } Import the construction to $(M_1,\xi_1)$ by considering the $\alpha_1$-Hamiltonian function
   \[h : M_1\to\R,\quad \quad \quad x\mapsto G(\pi_1(x)).\]
   Therefore the same arguments as in the proof of Proposition \ref{prop : lower bound} shows that for all $T\in[0,1]$
   \[\spec^{\alpha_1}(\phi_h^T)=\left\{T(\pi r^2-\varepsilon)+2\pi m\ |\ m\in\Z\right\}\bigcup\left\{2\pi m\ |\ m\in\Z\right\},\]
   where $\phi_h^T\in\conto(M_1,\xi_1)$.\\
   \textbf{Step 4 : } Therefore arguing as for Proposition \ref{prop : lower bound} we get that $c(\tphi_T)=T(\pi r^2-\varepsilon)$ for all $T\in[0,\delta]$ where $\delta=\min\{1, \frac{2\pi}{\pi r^2-\varepsilon}\}$ and $\tphi_T:=[(\phi_h^{tT})]$. Since by assumption $R\leq\sqrt{2}$ we deduce that $\delta=1$ and thus $c(\tphi_1)=\pi r^2-\varepsilon$ which implies that $C_{\alpha_1}(B_1(r))\geq\pi r^2-\varepsilon$. Taking the limit when $\varepsilon$ tends to $0$ yields $C_{\alpha_1}(B_1(r))\geq\pi r^2$.\\
   \textbf{Step 5 : } It remains now to apply Lemma \ref{lem : capacite croissante} to deduce that 
   \[C_{\alpha_k}(B_k(r))\geq C_{\alpha_1}(B_1(r))\geq \pi r^2.\]
\end{proof}

\subsection{Computation of the upper bound}
To finish the proof of Theorem \ref{thm : contact capacity for strongly orderable prequantizations} it thus remains to prove the upper bound that we restate below.
 \begin{prop}\label{prop : upper bound for strongly orderable prequantization}
    If $\Psi : (\mathbb{B}^{2n}(\lambda R),\omega_0)\hookrightarrow (W,\omega)$ is a symplectic embedding for some $R>0$ and $\lambda>\sqrt{3}$ then  $C_{\alpha_k}(B_k(r))\leq\pi r^2$ for all $r\in(0,R]$ where $B_k(r):=\pi_k^{-1}(\Psi(\mathbb{B}^{2n}(r)))$.
\end{prop}
This proposition will be an adaptation of the proof of Proposition \ref{prop : lower bound} together with the following Lemma which is an analogue of Theorem \ref{thm : capacite-energie}. First let us remark that 
\begin{equation}\label{eq : strict}\begin{aligned}\widetilde{\conto}^{\mathbb{T}^1}(M_k,\xi_k)&:=\left\{ \tpsi\in\widetilde{\conto}(L_k^{2n+1},\xi_k)\ \left|\
        \parbox{5.5cm}{there exists $(\psi_t)$ representing $\tpsi$
    such that $\phi_{\alpha_k}^s\circ\psi_t=\psi_t\circ\phi_{\alpha_k}^s$ for all $s\in\R$ and $t\in[0,1]$}\right.\right\}\\
    &=\left\{ \tpsi\in\widetilde{\conto}(L_k^{2n+1},\xi_k)\ \left|\
        \parbox{5cm}{there exists $(\psi_t)$ representing $\tpsi$
    such that $\psi_t^*\alpha=\alpha$ for all $t\in[0,1]$}\right.\right\}.
    \end{aligned}
    \end{equation}

\begin{lem}\label{lem : capacite-energie pour strongly orderable}
   Let $U$ be an open subset of $M_k$ and $\tpsi\in\widetilde{\conto}^{\mathbb{T}^1}(M_k,\xi_k)$. If $\Pi(\tpsi)$ $\alpha_k$-displaces $U$ then $C_{\alpha_k}(U)\leq C_+^{\alpha_k}(\tpsi^{-1})+C_+^{\alpha_k}(\tpsi).$
\end{lem}

\begin{proof}
    Let $\tphi\in\widetilde{\contoc}(U,\xi_k)$. Then thanks to Proposition \ref{prop : displacement} $C_+^{\alpha_k}(\tpsi\tphi)=C_+^{\alpha_k}(\tpsi)$. Moreover since $\tpsi\in\widetilde{\conto}^{\mathbb{T}^1}(M_k,\xi_k)$ it implies that $\tpsi^{-1}\in\widetilde{\conto}^{\mathbb{T}^1}(M_k,\xi_k)$. Hence Corollary \ref{lem : inegalite triangulaire strongly orderable prequantization} together with \eqref{eq : strict} ensure that 
    \[C_+^{\alpha_k}(\tphi)=C_+^{\alpha_k}\left(\tpsi^{-1}\cdot(\tpsi\tphi)\right)\leq C_+^{\alpha_k}(\tpsi^{-1})+C_+^{\alpha_k}(\tpsi\tphi)=C_+^{\alpha_k}(\tpsi^{-1})+C_+^{\alpha_k}(\tpsi)\]
    which concludes the proof. 
\end{proof}

\begin{proof}[Proof of Proposition \ref{prop : upper bound for strongly orderable prequantization}]Let $r\in(0,R]$, $\delta>0$ be small enough so that $\sqrt{3+2\delta}<\lambda$.\\
\textbf{Step 1 : } Consider the Hamiltonian function $\widetilde{K}$ supported in $\mathbb{B}^{2n}(\lambda R)$ constructed in the proof of Proposition \ref{prop : lower bound} \begin{enumerate}
\item $\widetilde{K}$ take values in $[0,r^2(\pi+\delta)]$ 
\item its time $1$ Hamiltonian flow displaces $B_k(r)$.
\end{enumerate}

\noindent
\textbf{Step 2 : } Import the construction to $(W,\omega)$ using $\Psi$, i.e. 
\[G:W\to\R \quad  \quad 
    x\mapsto \left\{
    \begin{array}{ll}
       H(\Psi^{-1}(x)) & \text{ if } x\in \Psi(\mathbb{B}^{2n}(r))\\
        0 & \text{ otherwise}
    \end{array}.
\right. \]
\textbf{Step 3 : } Import the construction to $(M_k,\xi_k)$ using Lemma \ref{lem : contact lift}. More precisely the $\alpha_k$-Hamiltonian function 
$h : M_k\to\mathbb{R}_{\geq 0}$, $x\mapsto G(\pi_k(x))$ generates a flow of $\mathbb{T}^1$-equivariant contactomorphisms $(\phi_h^t)\subset\conto(M_k,\xi_k)$. Moreover
\begin{enumerate}
    \item $h\geq 0$ and $\max h\leq r^2(\pi+\delta)$
    \item $\phi_h^1(B_k(r))\cap B_k(r)=\emptyset$, where $B_k(r):=\pi_k^{-1}(\Psi(\mathbb{B}^{2n}(r)))$.
\end{enumerate}
\noindent
\textbf{Step 4 : } Hence thanks to Lemma \ref{lem : capacite-energie pour strongly orderable}
\[C_{\alpha_k}(B_k(r))\leq C_+^{\alpha_k}(\widetilde{\phi}^{-1})+C_+^{\alpha_k}(\widetilde{\phi})\]
where $\widetilde{\phi}:=[(\phi_h^t)]\in\widetilde{\conto}^{\mathbb{T}^1}(M_k,\xi_k)$. Finally, since the $\alpha_k$-Hamiltonian function $M_k\to\R_{\leq 0},\ x\mapsto -h(x)$, generates the contact isotopy $(\phi_h^t)^{-1}$ which represents $\widetilde{\phi}^{-1}$, Proposition \ref{prop : spec selector strongly orderable} implies that 
\[ C_+^{\alpha_k}(\widetilde{\phi}^{-1})+C_+^{\alpha_k}(\widetilde{\phi}) \leq \int_0^1\max h\ud t-\int_0^1 \min h\ud t\leq r^2(\pi+\delta).\]
Letting $\delta$ goes to $0$ finishes the proof.

\end{proof}
%\newpage

\bibliographystyle{amsplain}
\bibliography{paa} 

\end{document}